\documentclass[12pt]{amsart}
\usepackage{graphicx}
\usepackage{amssymb}
\usepackage{amsfonts}
\usepackage{amsmath}
\usepackage{array}
\usepackage{rotating}
\usepackage{lscape}

\usepackage[all]{xy}

\allowdisplaybreaks

\newtheorem{example}{Example}[section]

\newtheorem{theorem}[example]{Theorem}

\newtheorem{corollary}[example]{Corollary}
\newtheorem{conjecture}[example]{Conjecture}

\newtheorem{proposition}[example]{Proposition}

\newtheorem{lemma}[example]{Lemma}

\def\Proof{\noindent \it Proof -- \rm}
\def\qed{\hspace{3.5mm} \hfill \vbox{\hrule height 3pt depth 2 pt width 2mm}
\bigskip}

\def\vect{{\rm Vect\,}}

\def\Sym{{\bf Sym}}
\def\NCSF{{\bf Sym}}

\def\PP{\mathring{\mathcal P}}
\def\GP{{\mathcal P}}   

\def\Std{{\rm Std}}

\def\Eu{{\mathcal E}^\uparrow}
\def\Ed{{\mathcal E}^\downarrow}

\def\<{\langle}
\def\>{\rangle}

\def\C{\operatorname{\mathbb C}}

\def\SG{{\mathfrak S}}

\def\tensor{\otimes}

\def\Des{\operatorname{Des}}

\def\dim{{\rm dim}}
\def\projSym#1{\mathfrak{proj}(#1)} 

\def\shuff#1#2{\mathbin{
\hbox{\vbox{ \hbox{\vrule \hskip#2 \vrule height#1 width 0pt
}%
\hrule}%
\vbox{ \hbox{\vrule \hskip#2 \vrule height#1 width 0pt
\vrule }%
\hrule}%
}}}

\def\shuf{{\mathchoice{\shuff{7pt}{3.5pt}}%
{\shuff{6pt}{3pt}}%
{\shuff{4pt}{2pt}}%
{\shuff{3pt}{1.5pt}}}}%
\def\shuffle{\,\shuf\,}

\def\MR{{\bf MR}}
\def\bA{{\bar A}}

\def\BSym{{\bf BSym}}

\def\div{\mid}
\def\ndiv{\nmid}
\newcommand{\Zr}[1][r]{\zeta^{(#1)}}

\usepackage{boxedminipage}
\usepackage[usenames]{color}
\newenvironment{NOTE}[1][NOTE]
 {\bigskip\begin{center}\color{red}\begin{boxedminipage}{4.5in}\setlength{\parindent}{1em}\noindent\color{blue}\textbf{#1. }}
 {\end{boxedminipage}\end{center}\bigskip}

\def\dwna{\mathord{\downarrow}}
\def\ordr{<}
\def\ordreq{\leq}

\def\epr{{e'}^{(r)}}
\def\er{{e}^{(r)}}

\def\gf#1#2{\genfrac{}{}{0pt}{}{#1}{#2}}

\title{Representation theory of the higher order peak algebras}

\author[J.-C.~Novelli,  F. Saliola, J.-Y.~Thibon]%
{Jean-Christophe Novelli, Franco Saliola, Jean-Yves Thibon}

\address{Institut Gaspard Monge, Universit\'e Paris-Est
Marne-la-Vall\'ee, \\ 5 Boulevard Descartes \\ Champs-sur-Marne \\
77454 Marne-la-Vall\'ee cedex 2 \\ France}

\email[Jean-Christophe Novelli]{novelli@univ-mlv.fr (corresponding author)}
\email[Franco Saliola]{saliola@gmail.com}
\email[Jean-Yves Thibon]{jyt@univ-mlv.fr} 
\date{\today}

\begin{document}

\begin{abstract}
The representation theory (idempotents, quivers, Cartan invariants and Loewy
series) of the  higher order unital peak algebras is investigated.
On the way, we obtain new interpretations and generating functions for the
idempotents of descent algebras introduced in [F. Saliola, J. Algebra 320
(2008) 3866.]
\end{abstract}

\maketitle

\section{Introduction}

A \emph{descent} of a permutation $\sigma\in\SG_n$ is an index $i$
such that $\sigma(i)>\sigma(i+1)$. A descent is a \emph{peak} if moreover
$i>1$ and $\sigma(i)>\sigma(i-1)$.
The sums of permutations with a given descent set span a subalgebra of the
group algebra, the \emph{descent algebra} $\Sigma_n$.
The \emph{peak algebra} $\PP_n$ of $\SG_n$ is a subalgebra of its descent
algebra, spanned by sums of permutations having the same peak set. This
algebra has no unit.

The direct sum of the peak algebras is a Hopf subalgebra of the
direct sum of all descent algebras, which can itself be identified with
$\Sym$, the Hopf algebra of noncommutative symmetric functions~\cite{NCSF1}.
Actually, in~\cite{BHT} it was shown that most of the results on
the peak algebras can be deduced from the case $q=-1$ of a $q$-identity
of~\cite{NCSF2}. Specializing $q$ to other roots of unity, Krob and the third
author introduced and studied \emph{higher order peak algebras} in~\cite{KT}.
Again, these are non-unital. 

In~\cite{ABN}, it has been shown that the peak algebra of $\SG_n$ can be naturally
extended to a unital algebra, which is obtained as a homomorphic image of the
descent algebra of the hyperoctahedral group $B_n$. This construction has
been extended in \cite{ANT}. It is shown there that unital versions
of the higher order peak algebras 
can be obtained as homomorphic
images of the \emph{Mantaci-Reutenauer algebras} of type $B$.

Our purpose here is to investigate the representation theory of the unital
higher order peak algebras.  The classical case has been worked out
in \cite{ANO}. In this reference, idempotents for the peak algebras were
obtained from those of the descent algebras of type $B$ constructed in
\cite{BB}. 

To deal with the general case, we need a different construction of
idempotents. It turns out that the recursive algorithm introduced
in~\cite{Sal} for idempotents of descent algebras can be adapted to higher
order peak algebras.

In order to achieve this, we need a better understanding of the idempotents
generated by the algorithm of \cite{Sal}. 
Interpreting them as noncommutative symmetric functions, we find that in type
$A$, these idempotents are associated with a known family of Lie idempotents,
the so-called \emph{Zassenhaus idempotents}, by the construction
of~\cite{NCSF2}. We then show that similar Lie idempotents can be defined in
type $B$ as well, which yields a simple generating function in terms of
noncommutative symmetric functions of type $B$.

This being understood, we obtain complete families of orthogonal idempotents
for the higher order peak algebras, which can be described either by
recurrence relations as in \cite{Sal} or by generating series of
noncommutative symmetric functions.

Finally, we make use of these idempotents to study the quivers, Cartan
invariants, and the Loewy series of the unital higher order peak algebras.

{\footnotesize {\it Acknowledgments.-}
This work has been partially supported by Agence Nationale de la Recherche,
grant ANR-06-BLAN-0380.
The authors are grateful to the contributors of both the MuPAD and Sage
projects~\cite{Sage}, and especially to those of the combinat package, for
providing the development environment for this research (see~\cite{HT} for an
introduction to MuPAD-Combinat).
}

\section{Notations and background}

\subsection{Noncommutative symmetric functions}

We will assume familiarity with the standard notations of the theory of
noncommutative symmetric functions~\cite{NCSF1} and with the main results of
\cite{KT,ANT}. We recall here only a few essential definitions.

The Hopf algebra of noncommutative symmetric functions is denoted by $\Sym$,
or by $\Sym(A)$ if we consider the realization in terms of an auxiliary
alphabet $A$. Linear bases of $\Sym_n$ are labelled by compositions
$I=(i_1,\ldots,i_r)$ of $n$ (we write $I\vDash n$).
The noncommutative complete and elementary functions are denoted
by $S_n$ and $\Lambda_n$, and $S^I=S_{i_1}\cdots S_{i_r}$.
The ribbon basis is denoted by $R_I$.
The \emph{descent set} of $I$ is
$\Des(I) = \{ i_1,\ i_1+i_2, \ldots , i_1+\dots+i_{r-1}\}$.
The \emph{descent composition} of a permutation $\sigma\in\SG_n$
is the composition $I=D(\sigma)$ of $n$ whose descent set is the descent set
of $\sigma$.

\subsection{The Mantaci-Reutenauer algebra of type $B$}

We denote by $\MR$ the free product $\Sym\star\Sym$ of two copies of the Hopf
algebra of noncommutative symmetric functions \cite{MR}. That is, $\MR$ is the
free associative algebra on two sequences $(S_n)$ and $(S_{\bar n})$
($n\ge 1$).
We regard the two copies of $\Sym$ as noncommutative symmetric functions on
two auxiliary alphabets: $S_n=S_n(A)$ and $S_{\bar n}=S_n(\bA)$.
We denote by $F\mapsto \bar F$ the involutive {\it anti}automorphism which
exchanges $S_n$ and $S_{\bar n}$.
The bialgebra structure is defined by the requirement that the series
\begin{equation}
\sigma_1=\sum_{n\ge 0}S_n \ \text{and}\ \bar\sigma_1=\sum_{n\ge 0}S_{\bar n}
\end{equation}
are grouplike.
The internal product of $\MR$ can be computed from the splitting formula
\begin{equation}
\label{splitting}
(f_1\dots f_r)*g = \mu_r\cdot (f_1\otimes\dots\otimes f_r)*_r \Delta^r g\,,
\end{equation}
where $\mu_r$ is $r$-fold multiplication, and $\Delta^r$ the iterated
coproduct with values in the $r$-th tensor power, and the conditions:
$\sigma_1$ is neutral,
$\bar\sigma_1$ is central, and
$\bar\sigma_1*\bar\sigma_1=\sigma_1$.

\subsection{Noncommutative symmetric functions of type $B$}

Noncommutative symmetric functions of type $B$ were introduced in~\cite{Chow}
as the right $\Sym$-module $\BSym$ freely generated by another sequence
$(\tilde S_n)$ ($n\ge 0$, $\tilde S_0=1$) of homogeneous elements, with
$\tilde\sigma_1$ grouplike. This is a coalgebra, but not an algebra.
It is endowed with an internal product, for which each homogeneous
component $\BSym_n$ is anti-isomorphic to the descent algebra of $B_n$.

It should be noted that with this definition, the restriction of the internal
product of $\BSym$ to $\Sym$ is not the internal product of $\Sym$. To remedy
this inconvenience, we use a different realization of $\BSym$.
We embed $\BSym$ as a sub-coalgebra and sub-$\Sym$-module of $\MR$ as follows.
Define, for $F\in \Sym(A)$,
\begin{equation}
F^\sharp = F(A|\bar A) = F(A-q\bar A)|_{q=-1}
\end{equation}
called the supersymmetric version, or superization, of $F$ \cite{NT-super}.
It is also equal to
\begin{equation}
\label{superization}
F^\sharp = F*\sigma_1^\sharp\,.
\end{equation}
Indeed, $\sigma_1^\sharp$ is grouplike,
and for $F=S^I$, the splitting formula gives
\begin{equation}
(S_{i_1}\cdots S_{i_r})*\sigma_1^\sharp
=\mu_r[(S_{i_1}\otimes\cdots\otimes S_{i_r})*
(\sigma_1^\sharp\otimes\cdots\otimes\sigma_1^\sharp)]=S^{I\sharp}\,.
\end{equation}
We have
\begin{equation}
\sigma_1^\sharp = \bar\lambda_1\sigma_1=\sum\Lambda_{\bar i}S_j \,.
\end{equation}
The element $\bar\sigma_1$ is central for the internal product, and
\begin{equation}
\bar\sigma_1 * F(A,\bar A) =  F(\bar A,A) = F*\bar\sigma_1\,.
\end{equation}
The basis element $\tilde S^I$ of $\BSym$, where $I=(i_0,i_1,\ldots,i_r)$ is a
type $B$-composition (that is, $i_0$ may be $0$), can be embedded as
\begin{equation}
\tilde S^I = S_{i_0}(A)S^{i_1i_2\cdots i_r}(A|\bar A)\,.
\end{equation}
We will identify $\BSym$ with its image under this embedding.

\subsection{Other notations}

For a partition $\lambda$, we denote by $m_i(\lambda)$ the multiplicity of
$i$ in $\lambda$ and set $m_\lambda:= \prod_{i\geq1}m_i(\lambda)!$.

The reverse refinement order on compositions is denoted by $\preceq$.
The nonincreasing rearrangement of a composition is denoted
by $I\dwna$. The refinement order on partitions is denoted by $\prec_{p}$:
$\lambda\prec_{p}\mu$ if $\lambda$ is finer than $\mu$, that is, each part
of $\mu$ is a sum of parts of $\lambda$.

\section{Descent algebras of type A}

\subsection{Principal idempotents}

In~\cite{Sal}, a recursive construction of complete sets of orthogonal
idempotents of descent algebras has been described. 
In~\cite{NCSF2}, one finds a general method for constructing such families
from an arbitrary sequence of Lie idempotents, as well as many remarkable
families of Lie idempotents.
It is therefore natural to investigate whether the resulting idempotents can
be derived from a (possibly known) sequence of Lie idempotents.
We shall show that it is indeed the case.

Let $P_n$ be the sequence of partitions of $n$ ordered in the following way:
first, sort them by decreasing length, then, for each length, order them by
reverse lexicographic order. We denote this order by $\leq$.
For example,  
\begin{equation}
P_5 = [11111,\ 2111,\ 311,\ 221,\ 41,\ 32,\ 5].
\end{equation}
Now, start with
\begin{equation}
e_{1^n} := \frac1{n!}{S_1^n},
\end{equation}
and define by induction
\begin{equation}
\label{eRecursion}
e_\lambda := \frac{1}{m_\lambda} S^\lambda * \left(S_n - \sum_{\mu<\lambda}
e_\mu\right).
\end{equation}

\begin{theorem}[\cite{Sal}]
The family $(e_\lambda)_{\lambda\vdash n}$ forms a complete system of
orthogonal idempotents for $\NCSF_n$.
\end{theorem}

Following \cite{NCSF2}, define the (left) Zassenhaus idempotents $\zeta_n$ by 
the generating series 
\begin{equation}
\sigma_1 =: \prod_{k\geq1}^{\leftarrow} e^{\zeta_k}
         = \cdots e^{\zeta_3} e^{\zeta_2} e^{\zeta_1}.
\end{equation}
For example,
\begin{equation}
S_1 = \zeta_1\ ,\quad S_2 = \zeta_2 + \frac12 \zeta_1^2\ , \quad
S_3 = \zeta_3 + \zeta_2\zeta1 + \frac16 \zeta_1^3,
\end{equation}
\begin{equation}
S_4 = \zeta_4 + \zeta_3\zeta_1 + \frac12 \zeta_2^2
    + \frac12 \zeta_2\zeta_1^2 + \frac1{24} \zeta_1^4,
\end{equation}
\begin{equation}
S_5 = 
  \zeta_{5}
+ \zeta_4\zeta_1
+ \zeta_3\zeta_2
+ \frac{1}{2} \zeta_3\zeta_1^2
+ \frac{1}{2} \zeta_2^2\zeta_1
+ \frac{1}{6} \zeta_2\zeta_1^3
+ \frac{1}{120} \zeta_1^5,
\end{equation}
\begin{align}
S_6 = 
&\
  \zeta_{6}
+ \zeta_5\zeta_1
+ \zeta_4\zeta_2
+ \frac{1}{2}\zeta_4\zeta_1^2
+ \frac{1}{2}\zeta_3^2
+ \zeta_3\zeta_2\zeta_1
\\ &
+ \frac{1}{6}\zeta_3\zeta_1^3
+ \frac{1}{6}\zeta_2^3
+ \frac{1}{4}\zeta_2^2\zeta_1^2
+ \frac{1}{24}\zeta_2\zeta_1^4
+ \frac{1}{720}\zeta_1^6 
\notag
\end{align}
so that
\begin{equation}
\zeta_1 = S_1\ ,\quad \zeta_2 = S_2 -\frac12 S^{11}\ , \quad
\zeta_3 = S_3 - S^{21} + \frac13 S^{111},
\end{equation}
\begin{equation}
\zeta_4 = S_4 - S^{31} - \frac12 S^{22} + \frac34 S^{211} + \frac14 S^{112}
      - \frac14 S^{1111},
\end{equation}
\begin{equation}
\zeta_5 = S_5 - S^{41} - S^{32} + S^{311} + S^{212} - \frac23 S^{2111}
           - \frac13 S^{1112} + \frac15 S^{11111},
\end{equation}
\begin{align}
\zeta_6 
= &\
  S_6
- S^{51}
- S^{42}
+ S^{411}
- \frac12 S^{33}
+ \frac12 S^{321}
+ S^{312}
- \frac56 S^{3111}
\\\notag &
+ \frac13 S^{222}
- \frac16 S^{2211}
+ \frac12 S^{213}
- \frac12 S^{2121}
- \frac23 S^{2112}
+ \frac{13}{24} S^{21111}
\\\notag & 
- \frac16 S^{1122}
+ \frac{1}{12} S^{11211}
- \frac16 S^{1113}
+ \frac16 S^{11121}
+ \frac{5}{24} S^{11112}
- \frac16 S^{111111}.
\end{align}

Note that in particular,
\begin{equation}
\label{SenZ}
S_n = \sum_{\lambda\vdash n}\ \ \frac{1}{m_\lambda}
          \zeta_{\lambda_1} \zeta_{\lambda_2} \dots \zeta_{\lambda_r}.
\end{equation}

For a composition $I=(i_1,\ldots,i_r)$, define as usual
$\zeta^I:=\zeta_{i_1}\dots \zeta_{i_r}$.
Since $\zeta_n\equiv S_n$ modulo smaller terms in the refinement order on
compositions, $\zeta^I\equiv S^I$ modulo smaller terms.
So the $\zeta^I$ family is unitriangular on the basis $S^J$, so it is a
basis of $\NCSF$. 

In the sequel, we shall need a condition for a product $S^I*\zeta^J$ to be
zero.
\begin{lemma}
\label{lem-SetZ}
Let $I$ and $J$ be two compositions of $n$. Then,
\begin{align}
S^I * \zeta^J
= 
\begin{cases}
0, & \text{if } J\dwna \not\prec_{p} I\dwna, \\
m_{I\dwna} \zeta^I & \text{if } J\dwna = I\dwna, \\
\sum_{K\dwna=J\dwna} c_{IJ}^K \zeta^K & otherwise, \\
\end{cases}
\end{align}
where $c_{IJ}^K$ is the number of ways of unshuffling $J$ into $p=\ell(I)$
subwords such that $J^{(l)}$ has sum $i_l$ and whose concatenation 
$J^{(1)} J^{(2)}\dots J^{(p)}$ is $K$.
\end{lemma}

\Proof
Since the Zassenhaus idempotents $\zeta_m$ are primitive, we have, thanks to
the splitting formula~(\ref{splitting}),
\begin{align}
\label{SZ}
S^I * \zeta^J
= 
\sum_{J^{(1)}, \cdots, J^{(p)}}
\left(S_{i_1} * \zeta^{J^{(1)}}\right) 
\cdots 
\left(S_{i_p} * \zeta^{J^{(p)}}\right)
\end{align}
where the sum ranges over all possible ways of decomposing $J$ into $p$
(possibly empty) subwords.

Since $S_{i_j} * \zeta^{J^{(j)}} = 0$ if $J^{(j)}$ is not a composition of
$i_j$, it follows that $S^I * \zeta^J = 0$ if $J\dwna \not\prec_{p} I\dwna$.
Moreover, if $J\dwna = I\dwna$, then
\begin{align}
S^I * \zeta^J
&=
\sum_{\sigma \in \SG_p}
\left(S_{i_1} * \zeta_{J_{\sigma(1)}}\right) 
\cdots 
\left(S_{i_p} * \zeta_{J_{\sigma(p)}}\right) 
= 
m_I \zeta^I.
\end{align}
If $J\dwna \prec_{p} I\dwna$ and $J\dwna \not= I\dwna$, then a term in the
r.h.s. of~(\ref{SZ}) is nonzero iff all $J^{(\ell)}$ are compositions of
$i_\ell$. In that case, we have
\begin{align}
S^I * \zeta^J
= \sum_{J^{(1)}, \dots, J^{(p)}} \zeta^{J^{(1)}} \cdots \zeta^{J^{(p)}}
= \sum_{J^{(1)}, \dots, J^{(p)}} \zeta^{J^{(1)}J^{(2)}\cdots J^{(p)}},
\end{align}
whence the last case.
\qed

\begin{theorem}
\label{SymIdempotentTheorem}
For all partitions $\lambda=(\lambda_1, \lambda_2, \dots, \lambda_k)$,
\begin{equation}
\label{eEpreformed}
e_\lambda = \frac1{m_\lambda} \ \zeta_{\lambda_1} \cdots \zeta_{\lambda_k}.
\end{equation}
\end{theorem}

\Proof
Let $e'_\lambda$ be the right-hand side of (\ref{eEpreformed}). We
will show that these elements satisfy the same induction as the
$e_\lambda$ (Equation~(\ref{eRecursion})).

From Lemma~\ref{lem-SetZ}, we have
\begin{align}
S^\lambda * e'_\lambda = m_\lambda e'_\lambda.
\end{align}
Now, using (\ref{SenZ}), we get
\begin{align}
m_\lambda e'_\lambda 
= S^\lambda * \left( S_n - \sum_{\mu\neq\lambda} e'_\mu \right)
= S^\lambda * \left( S_n - \sum_{\mu<\lambda} e'_\mu \right),
\end{align}
where the last equality follows again from Lemma~\ref{lem-SetZ}.
Hence, $e_\lambda = e'_\lambda$.
\qed

Note that, thanks to Lemma~\ref{lem-SetZ}, the induction formula for
$e_\lambda$ simplifies to
\begin{equation}
e_\lambda = \frac{1}{m_\lambda} S^\lambda 
 * \left(S_n - \sum_{\mu\prec_{p}\lambda} e_\mu\right).
\end{equation}

\subsection{A basis of idempotents}

As with any sequence of Lie idempotents, we can construct an idempotent
basis of $\Sym_n$ from the $\zeta_n$. Here, the principal idempotents
$e_\lambda$ are members of the basis, which leads to a simpler derivation
of the representation theory. 

We start with a basic lemma, easily derived from the splitting formula
(compare \cite[Lemma 3.10]{NCSF2}).
Recall that the radical of $(\Sym_n,*)$
is ${\mathcal R}_n = {\mathcal R}\cap\Sym_n$, where ${\mathcal R}$ is the
kernel of the commutative image $\Sym\rightarrow Sym$. 

\begin{lemma}
Denote by $\SG(J)$ the set of distinct rearrangements of a composition $J$.
Let $I=(i_1,\ldots,i_r)$ and $J=(j_1,\ldots,j_s)$ be two compositions of $n$.
Then,

\smallskip
(i) if $\ell(J)<\ell(I)$ then $\zeta^I*\zeta^J=0$.

\smallskip
(ii) if $\ell(J)>\ell(I)$ then
$\zeta^I*\zeta^J \in \vect\<\zeta^K\, : \, K\in\SG (J)\>\cap{\mathcal R}$.
More precisely,
\begin{equation}
\zeta^I*\zeta^J
 = \sum_{\gf{\scriptstyle J_1,\ldots, J_r}{\scriptstyle |J_k|=i_k}}
\< J\, ,\, J_1\shuffle \cdots \shuffle J_r \>\Gamma_{J_1}\cdots\Gamma_{J_r}
\end{equation}
where for a composition $K$ of $k$, $\Gamma_K:=\zeta_k*\zeta^K$.

\smallskip
(iii) if $\ell(J)=\ell(I)$, then $\zeta^I*\zeta^J\not= 0$
only for $J\in\SG (I)$, in which case $\zeta^I*\zeta^J=m_I\, \zeta^I$.
\end{lemma}

Note that the $\Gamma_K$ are in the primitive Lie algebra. This follows from
the *-multiplicativity of the coproduct: $\Delta(f*g)=\Delta(f)*\Delta(g)$,
see~\cite[Prop. 5.5]{NCSF1}.

\begin{corollary}
The elements
\begin{equation}
e_I=\frac1{m_I}\zeta^I\,,\quad I\vDash n\,,
\end{equation}
are all idempotents and form a basis of $\Sym_n$. This basis contains
in particular the principal idempotents $e_\lambda$.
\end{corollary}

\subsection{Cartan invariants}

By (iii) of the lemma, the indecomposable projective module
$P_\lambda=\Sym_n*e_\lambda$ contains the $e_I$ for $I\in\SG(\lambda)$. For
$I\not\in\SG(\lambda)$, (i) and (ii) imply that $e_I*e_\lambda$ is in
$\vect\<\zeta^K\, : \, K\in\SG (\lambda)\>$. Hence, this space coincides with
$P_\lambda$. So, we get immediately an explicit decomposition
\begin{equation}
\Sym_n=\bigoplus_{\lambda\vdash n}P_\lambda\,,\qquad
P_\lambda=\bigoplus_{I\in\SG(\lambda)}\C e_I\,.
\end{equation}
The Cartan invariants
\begin{equation}
c_{\lambda\mu}=\dim\,( e_\mu*\Sym_n*e_\lambda)
\end{equation}
are also easily obtained. The above space is spanned by the
\begin{equation}
e_\mu*e_I* e_\lambda = e_\mu* e_I\,,\quad I\in\SG(\lambda)\,.
\end{equation}
From (ii) of the lemma, this is the dimension of the space 
$[S^\mu(L)]_\lambda$, spanned by all symmetrized products of Lie polynomials
of degrees $\mu_1,\mu_2,\ldots$ formed from $\zeta_{i_1},\zeta_{i_2},\ldots$,
hence giving back the classical result of Garsia-Reutenauer \cite{GR}.

\subsection{Quiver and $q$-Cartan invariants (Loewy series)}

Still relying upon point (ii) of the lemma, we see that $c_{\lambda\mu}=0$ if
$\lambda$ is not finer than (or equal to) $\mu$, and that if $\mu$ is obtained
from $\lambda$ by adding up two parts $\lambda_i,\lambda_j$, $e_\mu* e_I=0$ if
$\lambda_i=\lambda_j$ and is a nonzero element of the radical otherwise. 

In~\cite{BL}, it is shown that the powers of the radical for the internal
product coincide with the lower central series of $\Sym$ for the external
product:
\begin{equation}
{\mathcal R}^{*j} =\gamma^j(\Sym)
\end{equation}
where $\gamma^j(\Sym)$ is the ideal generated by the commutators
$[\Sym,\gamma^{j-1}(\Sym)]$.
Hence, for $\lambda$ finer than $\mu$, $e_\mu* e_I$ is nonzero
modulo ${\mathcal R}^{*2}$ iff $\mu$ is obtained from $\lambda$
by summing two distinct parts. And more generally, 
$e_\mu* e_I$ is in ${\mathcal R}^{*k}$ and nonzero modulo
${\mathcal R}^{*k+1}$ iff $\ell(\lambda)-\ell(\mu)=k$.

Summarizing, we have

\begin{theorem}\cite{BL, Sal}
(i) In the quiver of $\Sym_n$, there is an arrow
$\lambda\rightarrow\mu$ iff $\mu$ is obtained from
$\lambda$ by adding two distinct parts.\\
(ii) The $q$-Cartan invariants are given by
\begin{equation}
c_{\lambda\mu}(q)=q^{\ell(\lambda)-\ell(\mu)}
\end{equation}
if  $\lambda$ is finer than (or equal to) $\mu$, and $c_{\lambda\mu}(q)=0$
otherwise.
\end{theorem}

\section{Descent algebras of type $B$}

\subsection{Preliminary lemmas on $\BSym$}

We begin by showing that in our realization of $\BSym$, Chow's map $\Theta$
(see~\cite[Section 3.4]{Chow}) corresponds to the left internal product by the
reproducing kernel $\sigma_1^\sharp$ of the superization map.
Chow's condition $\Theta(\tilde S_n)=S_n$ translates into the obvious equality
$\sigma_1^\sharp*\sigma_1=\sigma_1^\sharp$, the second condition
$\Theta(S_n(A))=S_n(2A)$ amounts to
\begin{equation}
\sigma_1^\sharp * \sigma_1^\sharp = (\sigma_1^\sharp)^2,
\end{equation}
which is an easy consequence of the splitting formula:
\begin{equation}
\sigma_1^\sharp * \sigma_1^\sharp = (\bar\lambda_1\sigma_1)*\sigma_1^\sharp 
=(\bar\lambda_1* \sigma_1^\sharp)\sigma_1^\sharp=(\sigma_1^\sharp)^2
\end{equation}
since
\begin{equation}
\label{barsharp}
\bar\lambda_1 * \sigma_1^\sharp = (\overline\lambda_1 * \sigma_1)
(\overline\lambda_1 * \overline\lambda_1)=\sigma_1^\sharp\,.
\end{equation}
Here we used the fact that left $*$-multiplication by $\bar\lambda_1$ is
an antiautomorphism.
Denoting by $\mu'$ as in \cite{Chow} the twisted product
\begin{equation}
\mu'(A\otimes B\otimes C) = (\overline\lambda_1*B) A C,
\end{equation}
we have:

\begin{lemma}
\begin{equation}
\sigma_1^\sharp * (FG)
= \mu' \left[ (\sigma_1^\sharp*F) \otimes \Delta(G)\right].
\end{equation}
\end{lemma}

\begin{proof}
\begin{equation}
\begin{split}
\sigma_1^\sharp * (FG)
&= \mu \left[ (\overline\lambda_1 \otimes \sigma_1)*(\Delta F\Delta G)\right]\\
&= \sum_{(F),(G)} \mu\left[ (\overline\lambda_1*F_1G_1) \otimes F_2G_2\right]\\
&= \sum (\overline\lambda_1*G_1)((\overline\lambda_1*F_1)F_2)G_2\\
&= \mu' \left[ (\sigma_1^\sharp*F) \otimes \Delta(G)\right].
\end{split}
\end{equation}
\end{proof}

\noindent
This is Chow's third condition, which completes the characterization of
$\Theta$.

\subsection{Idempotents in $\BSym$}

Define elements $\zeta_n \in \BSym$ by the generating series
\begin{align}
\label{BSymZ}
\sigma_1^\sharp =:
\left(e^{\zeta_1}e^{\zeta_2}e^{\zeta_3}\cdots\right)
\left(\cdots e^{\zeta_3}e^{\zeta_2}e^{\zeta_1}\right).
\end{align}
For example, collecting the terms of weights 1, 2 and 3, respectively, we have
\begin{align}
S_1^\sharp = 2 \zeta_1,
\quad
S_2^\sharp = 2 \zeta_2 + 2 \zeta_1^2,
\quad
S_3^\sharp = 2 \zeta_3 + 2 \zeta_2\zeta_1 + 2 \zeta_1\zeta_2
             + \frac43\zeta_1^3,
\end{align}
so that,
\begin{align}
\zeta_1 = \frac 12 S_1^\sharp,
\quad
\zeta_2 = \frac12 S^\sharp_2 - \frac 14 S^{11\sharp},
\quad
\zeta_3 = \frac12 S^\sharp_3 
        - \frac14 S^{21\sharp} - \frac14 S^{12\sharp}
        + \frac 16 S^{111\sharp}.
\end{align}
Note that the elements $\zeta_n$ are well-defined and that they are primitive. 
We shall use the notations
\begin{equation}
\left(e^{\zeta_1}e^{\zeta_2}\cdots\right) =: \Eu(\zeta)\,,\qquad
\left(\cdots e^{\zeta_2}e^{\zeta_1}\right) =: \Ed(\zeta).
\end{equation}

Next, define elements $\tilde \zeta_n \in \BSym$ by the generating series
\begin{align}
\label{BSymtZ}
\sigma_1 =:
  \left(\sum_{n\geq0} \tilde \zeta_n\right)
  \left(\cdots e^{\zeta_2}e^{\zeta_1}\right).
=: \tilde \zeta \Ed(\zeta).
\end{align}
For example,
\begin{gather}
S_1 = \zeta_1 + \tilde\zeta_1, \qquad
S_2 = \zeta_2 + \frac12\zeta_1^2 + \tilde\zeta_1\zeta_1 + \tilde\zeta_2, \\
S_3
    = \frac16\zeta_1^3 + \zeta_2\zeta_1 + \zeta_3 + \tilde\zeta_1\zeta_2 +
    \frac12\tilde\zeta_1\zeta_1^2 + \tilde\zeta_2\zeta_1 + \tilde\zeta_3,
\end{gather}
so that
\begin{gather}
\tilde\zeta_1 
    = S_1 - \frac12 S_1^\sharp,
\qquad
\tilde\zeta_2 
    = S_2 - \frac12 S_2^\sharp - \frac12 S_1S_1^\sharp + \frac38 S^{11\sharp}, 
\\
\tilde\zeta_3 
    = 
      S_3
    - \frac12 S_2 S_1^\sharp
    - \frac12 S_1 S_2^\sharp
    + \frac38 S_1 S^{11\sharp}
    - \frac12 S_3^\sharp
    + \frac14 S^{21\sharp} 
    + \frac12 S^{12\sharp} 
    - \frac{5}{16} S^{111\sharp}.
\end{gather}
Since $\sigma_1$ is grouplike, and since $e^{\zeta_n}$ is grouplike for all
$n\geq1$,  the series $\tilde \zeta$ is also
grouplike.

The next two lemmas describe some properties of the elements $\tilde \zeta_n$
and $\zeta_n$.

\begin{lemma}
The ordered exponentials are exchanged as follows:
\begin{equation}
\overline\lambda_1 * \Ed(\zeta) = \Eu(\zeta).
\end{equation}
In particular, $\overline\lambda_1 * \zeta_i = \zeta_i$ for all $i\geq0$.
\end{lemma}

\Proof
$\overline\lambda_1*\cdot$ is an antiautomorphism, so the left-hand side is
\begin{equation}
(\overline\lambda_1 * e^{\zeta_1})
(\overline\lambda_1 * e^{\zeta_2})
\dots
\end{equation}
Taking into account (\ref{barsharp}), and recalling that (\ref{BSymZ})
characterizes  the $\zeta_i$, we see that if we set
\begin{equation}
\zeta'_i := \overline\lambda_1 * \zeta_i,
\end{equation}
then
\begin{equation}
\sigma_1^\sharp =
  \left(e^{\zeta'_1}e^{\zeta'_2}\cdots\right)
  \left(\cdots e^{\zeta'_2}e^{\zeta'_1}\right),
\end{equation}
so that $\zeta'_i=\zeta_i$.
\qed

\begin{lemma}
\label{SigmaSharpIntProdTildeZeta}
For all $n\geq1$,
\begin{equation}
\sigma_1^\sharp * \tilde \zeta_n = 0.
\end{equation}
\end{lemma}

\Proof
By definition,
\begin{equation}
\begin{split}
\Eu(\zeta)\Ed(\zeta) = \sigma_1^\sharp
&= \sigma_1^\sharp * \sigma_1 \\
&= \sigma_1^\sharp * \left[\tilde \zeta \Ed(\zeta)\right] \\
&= (\overline\lambda_1\sigma_1) * (\tilde \zeta \Ed(\zeta))\\
&= \mu\left[ (\overline\lambda_1\otimes\sigma_1)
   * (\tilde \zeta \Ed(\zeta)\otimes \tilde \zeta \Ed(\zeta))\right]\\
&= (\overline\lambda_1 *  \tilde \zeta \Ed(\zeta))
   \ (\tilde \zeta \Ed(\zeta))\\
&= (\overline\lambda_1 *  \Ed(\zeta))
   \ (\overline\lambda_1 *  \tilde \zeta)
   \ \tilde \zeta 
   \ \Ed(\zeta) \\
&= \Eu(\zeta)
   \ (\overline\lambda_1 *  \tilde \zeta)
   \ (\sigma_1 * \tilde \zeta)
   \ \Ed(\zeta) \\
&= \Eu(\zeta)
   \ (\sigma_1^\sharp * \tilde \zeta)
   \ \Ed(\zeta)
\end{split} 
\end{equation}
so that $\sigma_1^\sharp * \tilde \zeta = 1$.
\qed

\begin{lemma}
\label{SigmaSharpIntProdZeta}
For all $n\geq1$,
\begin{equation}
\sigma_1^\sharp * \zeta_n = 2\zeta_n.
\end{equation}
\end{lemma}

\Proof
We have
\begin{equation}
\begin{split}
\sigma_1^\sharp * \zeta_n
&= (\overline\lambda_1\sigma_1) * \zeta_n\\
&= \mu\left[ (\overline\lambda_1\otimes\sigma_1)
             * (\zeta_n\otimes1+ 1\otimes \zeta_n)\right] \\
&= 2\zeta_n.
\end{split}
\end{equation}
\qed

\begin{proposition}
\label{BSymIntProdSandZ}
Let $I=(i_0,\ldots,i_p)$ be a $B$-composition of $n$ and let
$\lambda = (\lambda_0, \ldots, \lambda_k)$ be a $B$-partition of $n$.
\begin{align}
 \tilde S^I *
    \tilde \zeta_{\lambda_0} \zeta_{\lambda_1} \cdots \zeta_{\lambda_k}
= \begin{cases}
0, & \text{if } \lambda \not\preceq I\dwna, \\
\left(2^p \prod_{j\geq1} m_j!\right)
\tilde \zeta_{i_0} \zeta_{i_1} \cdots \zeta_{i_p}
& \text{if } \lambda = I\dwna,
\end{cases}
\end{align}
where $m_j$ is the multiplicity of $j$ in $(i_1, i_2, \ldots, i_p)$ (not
counting $i_0!$).
\end{proposition}

\Proof
The splitting formula yields
\begin{equation}
\begin{split}
\tilde S^I *
    \tilde \zeta_{\lambda_0} \zeta_{\lambda_1} \cdots \zeta_{\lambda_k}
&= \mu_p\left[
       (S_{i_0}\otimes S_{i_1}^\sharp\otimes\dots\otimes  S_{i_p}^\sharp)
       * \sum \tilde \zeta_{\alpha_0}^\sharp \zeta^{\alpha^{(0)}}
              \otimes\dots\otimes
              \tilde \zeta_{\alpha_p}^\sharp \zeta^{\alpha^{(p)}} \right] \\
&= \sum (S_{i_0} * \tilde \zeta_{\alpha_0}^\sharp \zeta^{\alpha^{(0)}})
         (S_{i_1}^\sharp * \tilde \zeta_{\alpha_1}^\sharp \zeta^{\alpha^{(1)}})
\cdots
        (S_{i_p}^\sharp * \tilde \zeta_{\alpha_p}^\sharp \zeta^{\alpha^{(p)}}).
\end{split}
\end{equation}
By Equation~\eqref{superization} and Lemma~\ref{SigmaSharpIntProdTildeZeta}, 
a summand is zero if any $\alpha_i>0$ for $i\geq1$, so that
\begin{equation}
\begin{split}
\tilde S^I * \tilde \zeta_{\lambda_0}
 \zeta_{\lambda_1} \cdots \zeta_{\lambda_k}
&= \sum (S_{i_0} * \tilde \zeta_{\alpha_0}^\sharp \zeta^{\alpha^{(0)}})
        (S_{i_1}^\sharp * \zeta^{\alpha^{(1)}}) \dots
        (S_{i_r}^\sharp * \zeta^{\alpha^{(r)}}).
\end{split}
\end{equation}
If a term is non-zero in this equation, then $\lambda \preceq I\dwna$. This
proves the first case.
By Equation~\eqref{superization} and Lemma~\ref{SigmaSharpIntProdZeta}, 
$S_i^\sharp*\zeta_i=2\zeta_i$, which proves the second case.
\qed

We are now in a position to give an explicit formula for the idempotents
of~\cite{Sal}.

\begin{theorem}
\label{BSymIdempotentTheorem}
For all $B$-partitions
$\lambda = (\lambda_0, \lambda_1, \ldots, \lambda_k)$ of $n$, define elements
$e_\lambda \in \BSym_n$ recursively by the formula
\begin{align}
e_\lambda = \frac1{2^k \prod_j m_j!} \tilde S^\lambda *
\left( S_n - \sum_{\mu < \lambda} e_\mu \right),
\end{align}
where $m_j$ is the multiplicity of $j$ in $(\lambda_1, \ldots, \lambda_k)$
(not counting $\lambda_0$!). Then
\begin{align}
\label{DefinitionBSymIdempotents}
e_\lambda = \frac1{\prod_j m_j!} \tilde \zeta_{\lambda_0} \zeta_{\lambda_1}
\cdots \zeta_{\lambda_k}.
\end{align}
\end{theorem}

\Proof
Let $e'_\lambda$ be the right-hand side of the above equation.
By Proposition \ref{BSymIntProdSandZ},
\begin{align}
\tilde S^\lambda * e'_\lambda
= \tilde S^\lambda * \frac1{\prod_j m_j!} \tilde \zeta_{\lambda_0}
  \zeta_{\lambda_1} \cdots \zeta_{\lambda_k}
= \left({2^k\prod_j m_j!}\right) e'_\lambda.
\end{align}
By Equation~\eqref{BSymtZ}, $S_n = \sum e'_\lambda$, where the sum
ranges over all $B$-partitions $\lambda$ of $n$. Together with the
above and Proposition~\ref{BSymIntProdSandZ}, we have
\begin{align}
\left(2^k \prod_j m_j!\right) e'_\lambda
= \tilde S^\lambda * e'_\lambda
= \tilde S^\lambda * \left( S_n - \sum_{\mu\neq\lambda} e'_\mu \right)
= \tilde S^\lambda * \left( S_n - \sum_{\mu<\lambda} e'_\mu \right).
\end{align}
Since the $e'_\lambda$ satisfy the same induction as the $e_\lambda$,
they are equal.
\qed

\section{Idempotents in the higher order peak algebras}

Let $q$ be a primitive $r$-th root of unity. 
We denote by $\theta_q$ the endomorphism of $\Sym$ defined by
\begin{equation}
\tilde f =\theta_q(f)=f((1-q)A)=f(A)*\sigma_1((1-q)A)\,.
\end{equation}
We denote by $\PP^{(r)}$ the image of $\theta_q$ and by $\GP^{(r)}$ the right
$\PP^{(r)}$-module generated by the $S_n$ for $n\ge 0$.
Note that $\PP^{(r)}$ is by definition a left $*$-ideal of $\Sym$.
For $r=2$, it is the classical peak ideal, and $\GP^{(2)}$ is the unital peak
algebra. For general $r$, $\PP^{(r)}$ is the higher order peak algebra of
\cite{KT} and $\GP^{(r)}$ is its unital extension defined in \cite{ANT}. These
objects depend only on $r$, and not on the choice of the primitive root of
unity.
Bases of $\GP^{(r)}$ can be labeled by $r$-peak compositions
$I=(i_0;i_1,\ldots,i_p)$, with at most one part $i_0$ divisible by $r$.

\subsection{The radical}
\label{PeakAlgebraRadical}

By definition, $\GP_n^{(r)}$ is a $*$-subalgebra of $\NCSF_n$.
The radical of $\NCSF_n$ consists of those elements whose commutative image is
zero (see \cite{NCSF2}, Lemma 3.10). The radical of  $\GP_n^{(r)}$ is
therefore spanned by the $S_{i_0}\cdot\theta_q(S^I-S^{I'})$ such that $I'$
is a permutation of $I$. Indeed, the quotient of $\GP_n^{(r)}$ by the span of
those elements is a semi-simple commutative algebra, the $*$-subalgebra of
$Sym_n$ spanned by the $p_\lambda$ ($\lambda\vdash n$) such that at most one
part of $\lambda$ is multiple of $r$. This special part will be denoted by
$\lambda_0$.

We denote by $P_n^{(r)}$ the subset of partitions of $n$ with at most
one part divisible by $r$. The simple $\GP_n^{(r)}$-modules, and the principal
idempotents, can therefore be labeled by  $P_n^{(r)}$.

\subsection{An induction for the idempotents}

Define a total order $\ordr$ on $P_n^{(r)}$ as follows: sort the partitions
by decreasing length, and sort partitions of the same length by reverse
lexicographic order. For example,
\begin{gather}
P_5^{(2)} = [11111,\ 2111,\ 311,\ 41,\ 23,\ 5], \\
P_7^{(2)} = [1111111,\ 211111,\ 31111,\ 4111,\ 2311,\ 511,\ 331,\ 61,\ 43,\
		25,\ 7].
\end{gather}

Now, set
\begin{equation}
\er_{1^n} := \frac1{n!}S_1^n,
\end{equation}
and define by induction
\begin{equation}
\label{erInduction}
\er_\lambda := \frac{1}{m_\lambda}\ T^\lambda * 
		\left(S_n - \sum_{\mu \ordr \lambda} \er_\mu\right)
\end{equation}
where $T_m = R_m$ if $r\div m$, $T_m = R_{r^{i} j}$ if $r \ndiv m$ and $m =
ir+j$ with $0<j<r$, and $T^\lambda = T_{\lambda_0} T_{\lambda_1} \cdots
T_{\lambda_p}$ for $\lambda=(\lambda_0;\lambda_1,\dots,\lambda_p) \in
P_n^{(r)}$. It follows from \cite[Cor. 3.17]{KT} and from the definition
of $\GP^{(r)}$ that $T^\lambda\in\GP^{(r)}$.

We want to prove that $(\er_\lambda)_{\lambda\in P_n^{(r)}}$
is a complete system of orthogonal idempotents for  $\GP^{(r)}$.

To this aim, we introduce the sequence of (left) Zassenhaus idempotents of
level $r$ $\zeta_n^{(r)}$ as the unique solution of the equation
\begin{equation}
\label{GenFunZetaR}
\sigma_1 =
\left(\sum_{p\geq0} \zeta_{pr}^{(r)}\right)
\prod_{i\geq1,\ r\,\ndiv\, i}^\leftarrow e^{\zeta_i^{(r)}}.
\end{equation}
Note that $\zeta^{(r)}_n = \zeta_n$ for $n<2r$.

For example, for $r=2$, 
\begin{gather}
\zeta_1^{(2)} = S_1\ ;\ \zeta_2^{(2)} = S_2 -\frac12 S^{11}\ ; \
\zeta_3^{(2)} = S_3 - S^{21} + \frac13 S^{111},\\
\zeta_4^{(2)} = S_4 - S^{31} + \frac12 S^{211} - \frac18 S^{1111},\\
\zeta_5^{(2)} = S_5 - S^{41} + \frac12 S^{311} - S^{23} + S^{221}
- \frac12 S^{2111} 
+ \frac12 S^{113} - \frac12 S^{1121} + \frac15 S^{11111}.
\end{gather}
And for $r=3$,
\begin{gather}
\zeta_1^{(3)} = S_1\ ;\ \zeta_2^{(3)} = S_2 -\frac12 S^{11}\ ; \
\zeta_3^{(3)} = S_3 - S^{21} + \frac13 S^{111}, \\
\zeta_4^{(3)} = S_4 - S^{31} - \frac12 S^{22} +\frac34S^{211} + \frac14 S^{112}
- \frac14 S^{1111}, \\
\zeta_5^{(3)} = S_5 - S^{41} - S^{32} + S^{311} + S^{212} - \frac23 S^{2111}
            - \frac13 S^{1112} + \frac15 S^{11111},
\end{gather}
\begin{align}
\notag
\zeta_6^{(3)} 
=&\
S_{6}
- S^{51}
- S^{42}
+ S^{411}
+ S^{312}
- \frac23 S^{3111}
+ \frac13 S^{222}
- \frac16 S^{2211}
\\ &
- \frac23 S^{2112}
+ \frac38 S^{21111}
- \frac16 S^{1122}
+ \frac{1}{12} S^{11211}
+ \frac{5}{24} S^{11112}
- \frac19 S^{111111}.
\end{align}

Define now for 
$\lambda=(\lambda_0; \lambda_1, \ldots, \lambda_k) \in P_n^{(r)}$,
\begin{equation}
{e'}_\lambda^{(r)} :=
 \frac1{m_\lambda} \
 \Zr_{\lambda_0} \Zr_{\lambda_1} \cdots \Zr_{\lambda_k}.
\end{equation}
We will show that ${e'}_\lambda^{(r)} = \er_\lambda$ for all $\lambda\in
P^{(r)}_n$. We begin with two lemmas.

\begin{lemma}
\begin{align}
\label{coprodZr}
\Delta\left(\Zr_n\right) =
\begin{cases}
1 \tensor \Zr_n + \Zr_n \tensor 1, & \text{if } r \ndiv n, \\
\sum_{i=0}^{n/r} \Zr_{ir} \tensor \Zr_{n-ir}, & \text{if } r \div n. \\
\end{cases}
\end{align}
\end{lemma}

\Proof
This means that the $\Zr_n$ are primitive if $r\div n$ and
that the generating series $\sum_{p\geq0}\Zr_{rp}$ is grouplike. If
we define new elements $Y_p$ by 
\begin{equation}\label{defY}
\sigma_1 =: \prod^\rightarrow_p e^{Y_{rp}}
            \prod_{r\,\ndiv\, i}^\leftarrow e^{Y_{i}},
\end{equation}
the standard argument showing that the Zassenhaus elements are primitive shows
as well that all the $Y_i$ are primitive. Now identify the first product in
the right-hand side with the generating series $\sum_{p\geq0}\Zr_{rp}$. Then
$\Zr_i = Y_i$ if $r \ndiv i$.
Since the exponential of a primitive element is grouplike and a product of
grouplike series is  group-like, both products in the right-hand side above
are grouplike. By identification, $\sum_{p\geq0}\Zr_{rp}$ is grouplike.
\qed

\begin{lemma}
\label{lem-TetZr}
Let $\lambda=(\lambda_0; \lambda_1, \ldots, \lambda_k) \in P_n^{(r)}$ and
$I = (i_0, i_1, \ldots, i_p)$ be an $r$-peak composition of $n$.
Then,
\begin{equation}
T^I * \Zr_{\lambda_0} \Zr_{\lambda_1} \cdots \Zr_{\lambda_k} 
= 
\begin{cases}
0, & \text{if } I\dwna \ordr \lambda, \\
m_{I}\ \Zr_{i_0} \Zr_{i_1} \cdots \Zr_{i_k}, & \text{if } I\dwna=\lambda.
\end{cases}
\end{equation}
\end{lemma}

\Proof
To simplify the notation, we let ${\Zr}^\lambda = \Zr_{\lambda_0}
\Zr_{\lambda_1} \cdots \Zr_{\lambda_k}$ for $\lambda\in P_n^{(r)}$.
If $F^I = F_{i_1} \cdots F_{i_p}$ with each $F_{i_j} \in \Sym_{i_j}$, the
splitting formula and \eqref{coprodZr} yield
\begin{align}
\label{Fepr}
F^I * {\Zr}^\lambda = 
\sum_{
\substack{
\lambda^{(1)} \vee \cdots \vee \lambda^{(p)} = (\lambda_1, \dots,
\lambda_k)\\
ra_1 + \cdots + ra_p = \lambda_0}
} 
\prod_{j=1}^p \left(F_{i_j} * \Zr_{ra_j} {\Zr}^{\lambda^{(j)}}\right)
\end{align}
where $\lambda_0$ is the part (possibly $0$) of $\lambda$ that is divisible
by $r$, and where $\alpha\vee\beta$ denotes the partition obtained by
reordering the concatenation of the partitions $\alpha$ and $\beta$.

Observe that since at most one $i_j$ is divisible by $r$, a product in the
above summation is $0$ if at least two of the partitions $\lambda^{(1)},
\ldots, \lambda^{(p)}$ are empty. If $\ell(I) > \ell(\lambda)$, then $k
\leq p - 2$, so this hypothesis is always satisfied. Thus, 
\begin{align}
\label{FeprZero}
F^I * {\Zr}^\lambda 
    = 0 \text{ if } \ell(I) > \ell(\lambda).
\end{align}

Suppose $I\dwna \ordreq \lambda$. Then, by definition of the order,
$\ell(I) \geq \ell(\lambda)$. Hence, for $I$ such that $I\dwna \ordreq
\lambda$ and $\ell(I) > \ell(\lambda)$, the result follows by taking $F^I =
T^I$ in \eqref{FeprZero}. So suppose instead that $\ell(I) = \ell(\lambda)$. 
By definition, $T^I = T_{i_1}\cdots T_{i_p}$, where $T_m = R_m$ if $r\div
m$ and $T_m = R_{r^{i} j}$ if $m = ir+j$ with $0<j<r$. Since $R_J$ can be
written as a linear combination of $S^K$ for which $J$ is a refinement of
$K$, it follows that $T^I$ is equal to $S^I$ plus a linear combination of
$S^K$ with $\ell(K) > \ell(I)$. By taking $F^I = S^K$ in \eqref{FeprZero},
it follows that $T^I * {\Zr}^\lambda = S^I * {\Zr}^\lambda$.

It remains to show that, for $I$ and $\lambda$ of the same length, $S^I *
{\Zr}^\lambda = m_I {\Zr}^\lambda$ if $I\dwna = \lambda$ and is $0$
otherwise. If $\lambda$ contains no part divisible by $r$, then it follows
from \eqref{Fepr} that if $S^I * {\zeta^{(r)}}^\lambda \neq 0$, we must
have $I\dwna = \lambda$, in which case $S^I * {\Zr}^\lambda = m_I
{\Zr}^\lambda$.

Suppose instead that $\lambda$ contains a part that is divisible by $r$.
Then each decomposition $(\lambda^{(1)}, \ldots, \lambda^{(p)})$ in
\eqref{Fepr} contains at least one $\lambda^{(j)} = \emptyset$. Thus, if
$I$ contains no part that is divisible by $r$, then $S^I * {\Zr}^\lambda =
0$. Otherwise, the part of $I$ that is divisible by $r$ is bounded by
$\lambda_0$. This implies that $\lambda \ordreq I\dwna$. Since we began by
assuming that $I\dwna \ordreq \lambda$, it follows that $I\dwna = \lambda$,
and the result follows from \eqref{Fepr} as before.
\qed

\begin{theorem}
For all partitions $\lambda=(\lambda_0; \lambda_1,
\ldots, \lambda_k) \in P_n^{(r)}$,
\begin{equation}
\label{DefinitionPeakIdempotents}
e_\lambda^{(r)} = {e'}_\lambda^{(r)} := 
 \frac1{m_\lambda} \
 \Zr_{\lambda_0} \Zr_{\lambda_1} \cdots \Zr_{\lambda_k}.
\end{equation}
\end{theorem}
\Proof
From the definition of $\Zr_m$, it follows that $S_n =
\sum_{\lambda\in P_n^{(r)}} \epr_\lambda$. Hence, by Lemma \ref{lem-TetZr},
\begin{align*}
m_\lambda \epr_\lambda 
= T^\lambda * e'_\lambda 
= T^\lambda * \left(S_n - \sum_{\mu \neq \lambda} \epr_\mu \right)
= T^\lambda * \left(S_n - \sum_{\mu \ordr \lambda} \epr_\mu \right).
\end{align*}
Thus, the elements $\epr_\lambda$ and $e^{(r)}_\lambda$ satisfy the same
induction equation \eqref{erInduction}.
\qed

\begin{theorem}
The family $(e^{(r)}_\lambda)_{\lambda\in P_n^{(r)}}$ forms a complete system
of orthogonal idempotents for $\GP^{(r)}_n$.
\end{theorem}

\begin{proof}
By construction, the $e^{(r)}_\lambda$ are in $\GP^{(r)}_n$. Their
identification with the ${e'}^{(r)}_\lambda$ shows that they are linearly
independent, and~(\ref{defY}) shows that they are orthogonal idempotents.
Indeed, the $Y_i$ are Lie idempotents of $\Sym$, and if we write partitions
$\mu$ of $n$ as $(\alpha;\beta)$, where the $\alpha_i$ are the parts divisible
by $r$ in increasing order and the $\beta_i$ the other parts in decreasing
order, then
\begin{equation}
Y_\mu = \frac1{m_\alpha m_\beta}Y^\alpha Y^\beta
\end{equation} 
is a complete set of orthogonal idempotents of $\NCSF_n$.
Hence, the $e^{(r)}_\lambda$, which are disjoint sums of the $Y_\mu$,
are orthogonal idempotents.
\end{proof}

\subsection{Cartan invariants}

Recall from subsection \ref{PeakAlgebraRadical} that the indecomposable
projective modules of $\GP^{(r)}_n$ can be labelled by $\mu\in P^{(r)}_n$.
For $\mu = (\mu_0; \mu_1, \ldots, \mu_p) \in P^{(r)}_n$, where $\mu_0$ is
is divisible by $r$, let $\bar\mu=(\mu_1,\ldots,\mu_p)$, and write
$\mu = (\mu_0; \bar\mu)$.

The description of the principal orthogonal idempotents in the previous
subsection shows that the dimension of the indecomposable projective module
labeled by $\mu$ is equal to the number of distinct permutations of
$\bar\mu$. We make the following conjecture:

\begin{conjecture}\label{conjcart}
The Cartan invariant $\dim(e^{(r)}_\nu * \GP^{(r)}_n * e^{(r)}_\mu)$ is the
number of permutations $\bar I$ of $\bar\mu$ for which the following
algorithm produces $\bar\nu$.
\begin{enumerate}
\item Compute the standardization $\tau=\Std(\bar I)$.
\item Replace the elements of the cycles of $\tau$ by the corresponding values in
      $\bar I$.
\item Take the sums of the values in each cycle, discarding those that are
      multiples of $r$.
\item The partition obtained by reordering these sums is $\bar\nu$.
\end{enumerate}
\end{conjecture}

\noindent
We will prove this for the classical case ($r=2$) in the next section.
We will also prove that the matrix of the Cartan invariants of
$\GP_n^{(2)}$ is obtained from that of $\BSym_n$ by merely selecting the
rows and columns labelled by $2$-peak partitions $P^{(2)}_n$. (That is, we
select for the type $B$ partitions $(\alpha_0;\alpha_1,\dots,\alpha_p)$
such that $\alpha_0$ is even and $\alpha_1,\dots,\alpha_p$ are odd.) We
know of no such simple description, even conjectural, for the general case.

\section{The classical peak algebras ($r=2$)}

In this section, we restrict our attention to the classical peak algebra
$\GP_n := \GP^{(2)}_n$. We will compute the $q$-Cartan matrix of $\GP_n$,
thus determining the quiver of $\GP_n$ and proving the conjecture at the
end of the previous section for $r=2$.

Recall that for type $B$ compositions or peak compositions $I = (i_0; i_1,
\ldots, i_p)$, we define $\bar I = (i_1,\ldots,i_p)$.

\subsection{$q$-Cartan invariants and quiver}

Let $\projSym{F} = F|_{\bar A =A}$ denote the projection of an arbitrary
element $F\in\MR$ onto $\Sym$. We will show that the projection of $\BSym$
is $\GP$, and also that $\GP$ can be identified with a subalgebra of
$\BSym$. This identification allows us to use known results about the
quiver and $q$-Cartan invariants of $\BSym$ to study $\GP$.

\begin{lemma}
\label{lemma:zetaprojections}
\begin{gather}
\projSym{\zeta_n}
=
\begin{cases}
\zeta^{(2)}_n, & \text{if $n$ is odd,} \\
0, & \text{if $n$ is even,}
\end{cases}
\qquad\text{and}\qquad
\projSym{\tilde\zeta_n}
=
\begin{cases}
0, & \text{if $n$ is odd,} \\
\zeta^{(2)}_n, & \text{if $n$ is even.}
\end{cases}
\end{gather}
In particular, if $I$ is a type $B$ composition, then
\begin{gather}
\projSym{e_I}
=
\begin{cases}
e^{(2)}_I, & \text{if $I$ is a $2$-peak composition,} \\
0, & \text{otherwise.}
\end{cases}
\end{gather}
\end{lemma}

\Proof 
Recall that the elements $\zeta_n, \tilde\zeta_n \in \BSym$ are defined as
the coefficients of $t^n$ in the following two series, respectively,
\begin{align}
\label{zetat}
\sigma_t^\sharp 
 &= \left( e^{\zeta_1t}e^{\zeta_2t^2}e^{\zeta_3t^3}\cdots \right)
 \left( \cdots e^{\zeta_3t^3}e^{\zeta_2t^2}e^{\zeta_1t}\right), \\
\label{tildezetat}
\sigma_t
 &= \left(\sum_{m\geq0} \tilde\zeta_mt^m\right)
 \left( \cdots e^{\zeta_3t^3}e^{\zeta_2t^2}e^{\zeta_1t}\right),
\end{align}
where
\begin{gather}
\sigma_t^\sharp := \sum_{n\geq0} S_n^\sharp t^n, 
\qquad 
\sigma_t := \sum_{n\geq0} S_n t^n.
\end{gather}

The projection of $\sigma_t^\sharp$ is
\begin{equation}
\projSym{\sigma_t^\sharp} = \projSym{\bar\lambda_t\sigma_t} =
\lambda_t\sigma_t,
\end{equation}
which satisfies
\begin{equation}
\projSym{\sigma_{-t}^\sharp}
    = \lambda_{-t}\sigma_{-t}
    = (\lambda_t\sigma_t)^{-1}
    = \projSym{\sigma_t^\sharp}^{-1}.
\end{equation}
Combined with Equation \eqref{zetat}, this yields the identity
\begin{align}
 &
 \left(
 e^{-\projSym{\zeta_1}t}e^{\projSym{\zeta_2}t^2}e^{-\projSym{\zeta_3}t^3}
 \cdots 
 \right)
 \left( \cdots
 e^{-\projSym{\zeta_3}t^3}e^{\projSym{\zeta_2}t^2}e^{-\projSym{\zeta_1}t}
 \right) 
 \\\notag
 &
 \qquad = \left(
 e^{-\projSym{\zeta_1}t}e^{-\projSym{\zeta_2}t^2}e^{-\projSym{\zeta_3}t^3}
 \cdots 
 \right)
 \left( \cdots
 e^{-\projSym{\zeta_3}t^3}e^{-\projSym{\zeta_2}t^2}e^{-\projSym{\zeta_1}t}
 \right),
\end{align}
from which it follows that
$\projSym{\zeta_{2m}} = 0$
for all $m\geq1$.

From Equations \eqref{zetat} and \eqref{tildezetat}, we have that
\begin{align}
\sigma_t
 &= \left(\sum_{m\geq0} \tilde\zeta_mt^m\right)
 \left( \cdots e^{\zeta_3t^3}e^{\zeta_2t^2}e^{\zeta_1t}\right) \\
&= 
 \bar\sigma_{-t}
 \left( e^{\zeta_1t}e^{\zeta_2t^2}e^{\zeta_3t^3}\cdots \right)
 \left( \cdots e^{\zeta_3t^3}e^{\zeta_2t^2}e^{\zeta_1t}\right),
\end{align}
and so a generating series for the $\tilde\zeta_m$ is given by
\begin{align}
\tilde\zeta(t) := \sum_{m\geq0} \tilde\zeta_m t^m
= 
\bar\sigma_{-t}
\left( e^{\zeta_1t}e^{\zeta_2t^2}e^{\zeta_3t^3}\cdots \right).
\end{align}
Denoting by $\projSym{\tilde\zeta}(t)$ the projection
of $\tilde\zeta(t)$, and remembering that
$\projSym{\zeta_{2m}}=0$,
\begin{align}
\projSym{\tilde\zeta}(-t) 
&= 
\sigma_{t}
\left(
e^{-\projSym{\zeta_1}t}e^{\projSym{\zeta_2}t^2}e^{-\projSym{\zeta_3}t^3}\cdots
\right)
= \projSym{\tilde\zeta}(t),
\end{align}
where the last equality follows from Equation \eqref{tildezetat}.
Thus, $\projSym{\tilde\zeta}(t)$ is even, which implies
that $\projSym{\tilde\zeta_{2m-1}} = 0$ for $m\geq1$.

Next, note that the image of the generating series in
Equation \eqref{tildezetat} is precisely the generating series of the
elements $\zeta^{(r)}_n$ given in Equation \eqref{GenFunZetaR}. 
Since this later generating series uniquely defines the elements
$\zeta^{(2)}_n$, it follows that
\begin{gather}
\projSym{\tilde\zeta_{2m}} = \zeta^{(2)}_{2m}
\quad\text{and}\quad
\projSym{\zeta_{2m+1}} = \zeta^{(2)}_{2m+1}.
\end{gather}

For the last assertion, recall that Equation
\eqref{DefinitionBSymIdempotents} defines $e_I\in\BSym$ for any type $B$
composition as
\begin{align}
e_I = 
\frac{1}{m_{\bar I}} \tilde\zeta_{i_0} \zeta_{i_1} \cdots \zeta_{i_p}.
\end{align}
Hence, if $I$ is not a $2$-peak composition (that is, if $i_0$ is not even
or if any of $i_1,\ldots,i_p$ are not odd), then $\projSym{e_I} = 0$.
And if $I$ is a $2$-peak composition,
then 
\begin{align}
\projSym{e_I} = \frac{1}{m_{\bar I}} \zeta^{(r)}_{i_0}
\zeta^{(r)}_{i_1} \cdots \zeta^{(r)}_{i_p} = e^{(2)}_I
\end{align}
where the last equality is just the definition of $e^{(2)}_I\in\GP$
(see Equation \eqref{DefinitionPeakIdempotents}). 
\qed

Using this result we can prove that the peak algebra $\GP_n$ is both a
quotient and a subalgebra of $\BSym_n$.

\begin{theorem} \
\label{thm:PeakInBSym}
(i) The projection of $\BSym$ onto $\Sym$ is $\projSym{\BSym} = \GP$.
\\
(ii) $\GP_n$ is isomorphic to the subalgebra 
$\varepsilon_n * \BSym_n * \varepsilon_n$
of $\BSym_n$, where
\begin{align}
\varepsilon_n = \sum_{\lambda\in P^{(2)}_n} e_\lambda \in \BSym_n.
\end{align}
\end{theorem}

\Proof
From Lemma \ref{lemma:zetaprojections}, the basis $(e_I)_I$ of $\BSym$,
where $I$ is a type $B$ composition, maps onto the basis $(e^{(2)}_J)_J$ of
$\GP$, where $J$ is a $2$-peak composition. This proves (i).

Since $\projSym{e_\lambda} = e^{(2)}_\lambda$ for $2$-peak partitions
$\lambda$,
\begin{align}
\projSym{\varepsilon_n} = \sum_{\lambda\in P^{(2)}_n} e^{(2)}_\lambda = 1 \in 
\GP_n.
\end{align}
Hence, $\projSym{\varepsilon_n * \BSym_n * \varepsilon_n} = \GP_n$.
We now only need to prove that the two algebras have the same dimension.

For any $2$-peak partition $\lambda$, the dimension of $\BSym_n *
e_\lambda$ is the number of rearrangements of $\bar\lambda$, which is also
the dimension of $\GP_n * e^{(2)}_\lambda$ (these statements follow
from the facts that the elements $e_I$ and $e^{(2)}_I$, one for each
$2$-peak composition $I$ such that $\bar I\dwna = \bar\lambda$, form bases
of $\BSym_n*e_\lambda$ and $\GP_n * e^{(2)}_\lambda$, respectively).
This implies that
\begin{align}
& \dim\left(\GP_n\right)
=
\dim\left(\GP_n * \sum_{\lambda\in P^{(2)}_n} e^{(2)}_\lambda\right)
=
\dim(\BSym_n * \varepsilon_n) \\
& \qquad \geq
\dim(\varepsilon_n * \BSym_n * \varepsilon_n) 
\geq
\dim(\projSym{\varepsilon_n * \BSym_n * \varepsilon_n})
=
\dim(\GP_n),
\end{align}
so that $\dim(\varepsilon_n * \BSym_n * \varepsilon_n) = \dim(\GP_n)$.
\qed

\begin{corollary}
Let us label the vertices of the quiver $Q_n$ of $\GP_n$ by odd
partitions $\bar\mu$ of $n,n-2,n-4,\ldots$. 
There are exactly $m>0$ arrows from $\bar\mu$ to $\bar\nu$ in $Q_n$ if and
only if $\bar\nu$ is obtained from $\bar\mu$ by:
deleting two unequal parts (in which case $m=1$); or 
merging three parts, of which at most two are equal (if no two of the
merged parts are equal, then $m=2$; if exactly two are equal, then $m=1$).
\end{corollary}

\Proof
By Theorem \ref{thm:PeakInBSym}, we can identify $\GP_n$ with the
subalgebra $\varepsilon_n * \BSym_n * \varepsilon_n$ of $\BSym$. Since
$\varepsilon_n$ is a sum of primitive orthogonal idempotents of $\BSym_n$,
it follows that the quiver $Q_n$ of $\GP_n$ is obtained from the
quiver of $\BSym_n$ by restricting to the vertices labelled by $2$-peak
partitions $\lambda$. The result follows immediately as the quiver of
$\BSym_n$ has already been computed \cite[Theorem~9.1]{Sal}.
%
\qed

From these results we can easily derive the results mentioned at the end of
the previous section for the $r=2$ case.

\begin{theorem} \
\begin{enumerate}
\item[\emph{(i)}]
The Cartan matrix of $\GP_n$ is obtained from that of $\BSym_n$ by
selecting the rows and columns labelled by the $2$-peak partitions
$P^{(2)}_n$.
\item[\emph{(ii)}]
For $r=2$, Conjecture \ref{conjcart} is true. Moreover, continuing with the
notation of Conjecture \ref{conjcart}, $I$ contributes
$q^{\frac{1}{2}(l(\bar\mu)-l(\bar\nu))}$ to the $q$-Cartan matrix of $\GP_n$.
\end{enumerate}
\end{theorem}

\Proof
By Lemma \ref{lemma:zetaprojections} and Theorem \ref{thm:PeakInBSym}, for
any $\nu,\mu\in P^{(2)}_n$ we have
\begin{align}
e^{(2)}_\nu * \GP_n * e^{(2)}_\mu
\cong 
e_\nu * \BSym_n * e_\mu,
\end{align}
which proves (i). 
For the first part of (ii), we recall that a combinatorial description of
the Cartan invariants of $\BSym_n$ was given in \cite{NBergeron1992}. If we
restrict that description to $2$-peak partitions, then the resulting
formulation is equivalent to that described in Conjecture \ref{conjcart}.
The final assertion follows from the description of the quiver of $\GP_n$
computed above. \qed

\subsection{Comparison with earlier works}

A basis of idempotents, and a complete set of orthogonal idempotents for
$\GP_n$, have been obtained in \cite{ANO}, relying on previous work on the
descent algebras of type $B$ \cite{BB}. Here is a short derivation of these
results, using our realization of type $B$ noncommutative symmetric
functions \cite{Chow}.

Let us recall the notation
\begin{equation}
\phi(t)= \log \sigma_t = \sum_{n\ge 1}\frac{\Phi_n}{n}t^n\,.
\end{equation}
For $r=2$,
\begin{equation}
\sigma_t^\sharp:=\sigma_t(A-q\bar A)|_{q=-1}=\bar\lambda_t\sigma_t
=\exp(\phi^\sharp(t))\,.
\end{equation}
We have already seen that
\begin{equation}
\sigma_t^\sharp * \sigma_t^\sharp 
= (\sigma_t^\sharp)^2
= \exp(2\phi^\sharp(t))\,,
\end{equation}
so that
\begin{equation}
E_\lambda
:=\sum_{I\downarrow=\lambda}
  \frac{\Phi^{I\sharp}}{2^{l(I)}i_1i_2\cdots i_{l(I)}}
\end{equation}
are orthogonal idempotents, summing to $\sigma_1^\sharp$. 

Following \cite[Def. 4.14]{Chow}, define
\begin{equation}
\eta(t):=\sigma_t\cdot[\sigma_t^\sharp]^{-\frac12}\,.
\end{equation}
Then, $\eta(1)$ is an idempotent.

Denote by $\tilde f = f|_{\bar A =A}$ the projection onto $\NCSF$ of an
element of $\MR$. Then,
\begin{equation}
\xi(t):=\tilde \sigma_t^\sharp =\lambda_t\sigma_t
\end{equation}
satisfies
\begin{equation}
\xi(-t)=\xi(t)^{-1}\,.
\end{equation}
Indeed,
\begin{equation}
\xi(-t)=\lambda_{-t}\sigma_{-t}=(\lambda_t\sigma_t)^{-1}\,,
\end{equation}
so that
\begin{equation}
\tilde\phi^\sharp(-t)=-\tilde\phi^\sharp(t)\,,
\end{equation}
that is, only the $\phi_n$ with $n$ odd  survive the projection onto $\GP$.
Moreover,
\begin{equation}
\begin{split}
\tilde\eta(-t)&=\sigma_{-t}(\lambda_{-t}\sigma_{-t})^{-\frac12}
=\sigma_{-t}(\lambda_t\sigma_t)^{\frac12}\\
&= [(\lambda_t\sigma_t)^{-\frac12}\lambda_t\sigma_t\lambda_{-t}]^{-1}
=[(\lambda_t\sigma_t)^{\frac12}\lambda_{-t}]^{-1}=\tilde\eta(t)\,,
\end{split}
\end{equation}
so that $\tilde\eta(-t)$ is even, and only the $\eta_{2k}$ survive.
Hence, if
$\lambda=(\lambda_0,\lambda_1,\ldots,\lambda_m)=(\lambda_0;\bar\lambda)$
runs over partitions of $n$ such that $\lambda_0$ is even (allowed to be 0)
and the other parts are odd,
\begin{equation}
\tilde E_\lambda
 = \eta_{\lambda_0}
   \sum_{J\downarrow\bar\lambda}
       \frac{\tilde\Phi^{I\sharp}}{2^{l(I)}i_1i_2\cdots i_{l(I)}}
\end{equation}
is a complete set of orthogonal idempotents of $\GP_n$, consisting of the
nonzero images of a complete set for $\BSym_n$. These idempotents coincide
with those of \cite{ANO}, but are different from ours. They have nevertheless
a similar structure.

\section{$q$-Cartan invariants for generalized peak algebras}

Conjecture \ref{conjcart} provided a conjectural description of the Cartan
invariants of the higher order peak algebras $\GP_n^{(r)}$. The following
presents their $q$-analogs: the coefficient of $q^k$ in row $\lambda$,
column $\mu$, is the multiplicity of the simple module $\lambda$ in the
$k$th slice of the descending Loewy series of the indecomposable projective
module $\mu$. (Here $q$ is an indeterminate, and not to be confused with
the primitive root of unity that $q$ denoted earlier.) The usual Cartan
invariants can be obtained from this $q$-analog by setting $q=1$. Note that
the $q$-Cartan matrix also encodes the quiver of the algebra since the
single powers of $q$ in the matrix correspond to the arrows of the quiver.

Also note that for $r\geq n$, the $q$-Cartan matrix of $\GP_n^{(r)}$ is the
same, up to indexation, as the $q$-Cartan matrix of $\NCSF_n$. Below we
write partitions in the order opposite to the order $\ordr$ defined above,
and the zero entries of the matrices are represented by dots to enhance
readability. 

\subsection{$q$-Cartan matrices for unital peak algebras ($r=2$)}

\begin{equation}
C_2^{(2)} =
\left(\begin{array}{cc}
1 & . \\
. & 1
\end{array}\right)
\end{equation}

\begin{equation}
C_3^{(2)} =
\left(\begin{array}{ccc}
1 & . & . \\
. & 1 & . \\
. & . & 1
\end{array}\right)
\end{equation}

\begin{equation}
C_4^{(2)} =
\left(\begin{array}{cccc}
1 & q & . & . \\
. & 1 & . & . \\
. & . & 1 & . \\
. & . & . & 1
\end{array}\right)
\end{equation}

\begin{equation}
C_5^{(2)} =
\left(\begin{array}{cccccc}
1 & . & . & q & . & . \\
. & 1 & . & . & . & . \\
. & . & 1 & q & . & . \\
. & . & . & 1 & . & . \\
. & . & . & . & 1 & . \\
. & . & . & . & . & 1
\end{array}\right)
\end{equation}

\begin{equation}
C_6^{(2)} =
\left(\begin{array}{cccccccc}
1 & . & q & q & . & q^{2} & . & . \\
. & 1 & . & . & . & .     & . & . \\
. & . & 1 & . & . & q     & . & . \\
. & . & . & 1 & . & .     & . & . \\
. & . & . & . & 1 & q     & . & . \\
. & . & . & . & . & 1     & . & . \\
. & . & . & . & . & .     & 1 & . \\
. & . & . & . & . & .     & . & 1
\end{array}\right)
\end{equation}

\begin{equation}
C_7^{(2)} =
\left(\begin{array}{ccccccccccc}
1 & . & . & . & q & q & . & . & q^{2} & . & . \\
. & 1 & . & . & q & . & . & . & . & . & . \\
. & . & 1 & . & . & . & q & . & . & . & . \\
. & . & . & 1 & . & q & q & . & q^{2} & . & . \\
. & . & . & . & 1 & . & . & . & . & . & . \\
. & . & . & . & . & 1 & . & . & q & . & . \\
. & . & . & . & . & . & 1 & . & . & . & . \\
. & . & . & . & . & . & . & 1 & q & . & . \\
. & . & . & . & . & . & . & . & 1 & . & . \\
. & . & . & . & . & . & . & . & . & 1 & . \\
. & . & . & . & . & . & . & . & . & . & 1
\end{array}\right)
\end{equation}

\begin{equation}
C_8^{(2)} =
\left(\begin{array}{cccccccccccccc}
1 & q & q & . & q & q & . & 2 q^{2} & q^{2} & q^{2} & . & q^{3} & . & . \\
. & 1 & . & . & . & . & . & q       & .     & .     & . & .     & . & . \\
. & . & 1 & . & . & . & . & q       & q     & .     & . & q^{2} & . & . \\
. & . & . & 1 & . & . & . & .       & .     & .     & . & .     & . & . \\
. & . & . & . & 1 & . & . & q       & .     & .     & . & .     & . & . \\
. & . & . & . & . & 1 & . & .       & .     & q     & . & .     & . & . \\
. & . & . & . & . & . & 1 & .       & q     & q     & . & q^{2} & . & . \\
. & . & . & . & . & . & . & 1       & .     & .     & . & .     & . & . \\
. & . & . & . & . & . & . & .       & 1     & .     & . & q     & . & . \\
. & . & . & . & . & . & . & .       & .     & 1     & . & .     & . & . \\
. & . & . & . & . & . & . & .       & .     & .     & 1 & q     & . & . \\
. & . & . & . & . & . & . & .       & .     & .     & . & 1     & . & . \\
. & . & . & . & . & . & . & .       & .     & .     & . & .     & 1 & . \\
. & . & . & . & . & . & . & .       & .     & .     & . & .     & . & 1
\end{array}\right)
\end{equation}

\begin{equation}
C_9^{(2)} =
\left(\begin{array}{ccccccccccccccccccc}
1 & . & . & . & . & . & 2 q & q & . & . & . & . & 2 q^{2} & q^{2} & . & . & q^{3} & . & . \\
. & 1 & . & . & . & . & q & . & . & q & . & . & q^{2} & . & . & . & . & . & . \\
. & . & 1 & . & . & . & q & . & q & . & . & . & q^{2} & . & . & . & . & . & . \\
. & . & . & 1 & . & . & . & . & q & . & q & . & . & . & q^{2} & . & . & . & . \\
. & . & . & . & 1 & . & q & q & . & q & q & . & 2 q^{2} & q^{2} & q^{2} & . & q^{3} & . & . \\
. & . & . & . & . & 1 & . & . & . & . & . & . & . & . & . & . & . & . & . \\
. & . & . & . & . & . & 1 & . & . & . & . & . & q & . & . & . & . & . & . \\
. & . & . & . & . & . & . & 1 & . & . & . & . & q & q & . & . & q^{2} & . & . \\
. & . & . & . & . & . & . & . & 1 & . & . & . & . & . & . & . & . & . & . \\
. & . & . & . & . & . & . & . & . & 1 & . & . & q & . & . & . & . & . & . \\
. & . & . & . & . & . & . & . & . & . & 1 & . & . & . & q & . & . & . & . \\
. & . & . & . & . & . & . & . & . & . & . & 1 & . & q & q & . & q^{2} & . & . \\
. & . & . & . & . & . & . & . & . & . & . & . & 1 & . & . & . & . & . & . \\
. & . & . & . & . & . & . & . & . & . & . & . & . & 1 & . & . & q & . & . \\
. & . & . & . & . & . & . & . & . & . & . & . & . & . & 1 & . & . & . & . \\
. & . & . & . & . & . & . & . & . & . & . & . & . & . & . & 1 & q & . & . \\
. & . & . & . & . & . & . & . & . & . & . & . & . & . & . & . & 1 & . & . \\
. & . & . & . & . & . & . & . & . & . & . & . & . & . & . & . & . & 1 & . \\
. & . & . & . & . & . & . & . & . & . & . & . & . & . & . & . & . & . & 1
\end{array}\right)
\end{equation}

\subsection{$q$-Cartan matrices for unital peak algebras of order $r=3$}
\begin{equation}
C_3^{(3)} =
\left(\begin{array}{ccc}
1 & q & . \\
. & 1 & . \\
. & . & 1
\end{array}\right)
\end{equation}

\begin{equation}
C_4^{(3)} =
\left(\begin{array}{ccccc}
1 & . & . & q & . \\
. & 1 & . & . & . \\
. & . & 1 & q & . \\
. & . & . & 1 & . \\
. & . & . & . & 1
\end{array}\right)
\end{equation}

\begin{equation}
C_5^{(3)} =
\left(\begin{array}{ccccccc}
1 & . & q & q & . & q^{2} & . \\
. & 1 & . & q & . & . & . \\
. & . & 1 & . & . & q & . \\
. & . & . & 1 & . & . & . \\
. & . & . & . & 1 & q & . \\
. & . & . & . & . & 1 & . \\
. & . & . & . & . & . & 1
\end{array}\right)
\end{equation}

\begin{equation}
C_6^{(3)} =
\left(\begin{array}{cccccccccc}
1 & q & q & . & q & q^{2} & 2 q^{2} & . & q^{3} & . \\
. & 1 & . & . & . & .     & q       & . & .     & . \\
. & . & 1 & . & . & q     & q       & . & q^{2} & . \\
. & . & . & 1 & . & .     & .       & . & .     & . \\
. & . & . & . & 1 & .     & q       & . & .     & . \\
. & . & . & . & . & 1     & .       & . & q     & . \\
. & . & . & . & . & .     & 1       & . & .     & . \\
. & . & . & . & . & .     & .       & 1 & q     & . \\
. & . & . & . & . & .     & .       & . & 1     & . \\
. & . & . & . & . & .     & .       & . & .     & 1
\end{array}\right)
\end{equation}

\begin{equation}
C_7^{(3)} =
\left(\begin{array}{cccccccccccccc}
1 & . & q & . & . & q^{2} + q & q & q^{2} & . & q^{2} & q^{3} + q^{2} & . & q^{3} & . \\
. & 1 & . & . & . & q & . & . & q & . & q^{2} & . & . & . \\
. & . & 1 & . & . & q & . & q & . & . & q^{2} & . & . & . \\
. & . & . & 1 & . & q & q & . & q & q^{2} & 2 q^{2} & . & q^{3} & . \\
. & . & . & . & 1 & . & . & q & . & . & . & . & . & . \\
. & . & . & . & . & 1 & . & . & . & . & q & . & . & . \\
. & . & . & . & . & . & 1 & . & . & q & q & . & q^{2} & . \\
. & . & . & . & . & . & . & 1 & . & . & . & . & . & . \\
. & . & . & . & . & . & . & . & 1 & . & q & . & . & . \\
. & . & . & . & . & . & . & . & . & 1 & . & . & q & . \\
. & . & . & . & . & . & . & . & . & . & 1 & . & . & . \\
. & . & . & . & . & . & . & . & . & . & . & 1 & q & . \\
. & . & . & . & . & . & . & . & . & . & . & . & 1 & . \\
. & . & . & . & . & . & . & . & . & . & . & . & . & 1
\end{array}\right)
\end{equation}

{
\tiny
\begin{equation}
C_8^{(3)} =
\left(\begin{array}{cccccccccccccccccccc}
1 & . & . & . & q & q & . & q^{2} + q & . & . & . & q^{3} + 2 q^{2} & q^{2} & q^{3} + q^{2} & . & q^{3} & q^{4} + q^{3} & . & q^{4} & . \\
. & 1 & . & . & . & . & . & . & . & . & . & q & . & . & . & . & q^{2} & . & . & . \\
. & . & 1 & . & . & . & q & q & . & . & q & q^{2} & . & q^{2} & q^{2} & . & q^{3} & . & . & . \\
. & . & . & 1 & . & q & . & q & . & . & q & q^{2} & . & 2 q^{2} & . & . & q^{3} & . & . & . \\
. & . & . & . & 1 & . & . & q & . & . & . & q^{2} + q & q & q^{2} & . & q^{2} & q^{3} + q^{2} & . & q^{3} & . \\
. & . & . & . & . & 1 & . & . & . & . & . & . & . & q & . & . & . & . & . & . \\
. & . & . & . & . & . & 1 & . & . & . & . & q & . & . & q & . & q^{2} & . & . & . \\
. & . & . & . & . & . & . & 1 & . & . & . & q & . & q & . & . & q^{2} & . & . & . \\
. & . & . & . & . & . & . & . & 1 & . & . & q & q & . & q & q^{2} & 2 q^{2} & . & q^{3} & . \\
. & . & . & . & . & . & . & . & . & 1 & . & . & . & . & . & . & . & . & . & . \\
. & . & . & . & . & . & . & . & . & . & 1 & . & . & q & . & . & . & . & . & . \\
. & . & . & . & . & . & . & . & . & . & . & 1 & . & . & . & . & q & . & . & . \\
. & . & . & . & . & . & . & . & . & . & . & . & 1 & . & . & q & q & . & q^{2} & . \\
. & . & . & . & . & . & . & . & . & . & . & . & . & 1 & . & . & . & . & . & . \\
. & . & . & . & . & . & . & . & . & . & . & . & . & . & 1 & . & q & . & . & . \\
. & . & . & . & . & . & . & . & . & . & . & . & . & . & . & 1 & . & . & q & . \\
. & . & . & . & . & . & . & . & . & . & . & . & . & . & . & . & 1 & . & . & . \\
. & . & . & . & . & . & . & . & . & . & . & . & . & . & . & . & . & 1 & q & . \\
. & . & . & . & . & . & . & . & . & . & . & . & . & . & . & . & . & . & 1 & . \\
. & . & . & . & . & . & . & . & . & . & . & . & . & . & . & . & . & . & . & 1
\end{array}\right)
\end{equation}
}

\subsection{$q$-Cartan matrices for unital peak algebras of order $r=4$}

\begin{equation}
C_4^{(4)} =
\left(\begin{array}{ccccc}
1 & . & q & q^{2} & . \\
. & 1 & . & . & . \\
. & . & 1 & q & . \\
. & . & . & 1 & . \\
. & . & . & . & 1
\end{array}\right)
\end{equation}

\begin{equation}
C_5^{(4)} =
\left(\begin{array}{ccccccc}
1 & q & . & q^{2} & q & q^{2} & . \\
. & 1 & . & q & . & . & . \\
. & . & 1 & . & q & q^{2} & . \\
. & . & . & 1 & . & . & . \\
. & . & . & . & 1 & q & . \\
. & . & . & . & . & 1 & . \\
. & . & . & . & . & . & 1
\end{array}\right)
\end{equation}

\pagebreak[4]

\begin{equation}
C_6^{(4)} =
\left(\begin{array}{ccccccccccc}
1 & . & . & q & . & q^{2} + q & . & q^{3} & q^{2} & q^{3} & . \\
. & 1 & . & . & . & q & . & q^{2} & . & . & . \\
. & . & 1 & . & . & q & . & q^{2} & . & . & . \\
. & . & . & 1 & . & q & . & q^{2} & q & q^{2} & . \\
. & . & . & . & 1 & . & . & . & . & . & . \\
. & . & . & . & . & 1 & . & q & . & . & . \\
. & . & . & . & . & . & 1 & . & q & q^{2} & . \\
. & . & . & . & . & . & . & 1 & . & . & . \\
. & . & . & . & . & . & . & . & 1 & q & . \\
. & . & . & . & . & . & . & . & . & 1 & . \\
. & . & . & . & . & . & . & . & . & . & 1
\end{array}\right)
\end{equation}

\begin{equation}
C_7^{(4)} =
\left(\begin{array}{ccccccccccccccc}
1 & . & q & q & q^{2} & q & . & q^{2} & q^{3} & q^{3} + 2 q^{2} & . & q^{4} + q^{3} & q^{3} & q^{4} & . \\
. & 1 & . & . & . & q & q & . & . & 2 q^{2} & . & q^{3} & . & . & . \\
. & . & 1 & . & q & . & . & . & q^{2} & q & . & q^{2} & . & . & . \\
. & . & . & 1 & . & . & . & q & . & q^{2} + q & . & q^{3} & q^{2} & q^{3} & . \\
. & . & . & . & 1 & . & . & . & q & . & . & . & . & . & . \\
. & . & . & . & . & 1 & . & . & . & q & . & q^{2} & . & . & . \\
. & . & . & . & . & . & 1 & . & . & q & . & q^{2} & . & . & . \\
. & . & . & . & . & . & . & 1 & . & q & . & q^{2} & q & q^{2} & . \\
. & . & . & . & . & . & . & . & 1 & . & . & . & . & . & . \\
. & . & . & . & . & . & . & . & . & 1 & . & q & . & . & . \\
. & . & . & . & . & . & . & . & . & . & 1 & . & q & q^{2} & . \\
. & . & . & . & . & . & . & . & . & . & . & 1 & . & . & . \\
. & . & . & . & . & . & . & . & . & . & . & . & 1 & q & . \\
. & . & . & . & . & . & . & . & . & . & . & . & . & 1 & . \\
. & . & . & . & . & . & . & . & . & . & . & . & . & . & 1
\end{array}\right)
\end{equation}

{
\tiny
\begin{equation}
C_8^{(4)} =
\left(\begin{array}{ccccccccccccccccccccc}
1 & q & q & q & q^{2} & . & q & 2 q^{2} & q^{2} & . & 3 q^{3} & 2 q^{2} & q^{2} & q^{3} & 2 q^{4} & q^{4} + 4 q^{3} & . & q^{5} + 2 q^{4} & q^{4} & q^{5} & . \\
. & 1 & . & . & q & . & . & q & . & . & 2 q^{2} & q & . & . & q^{3} & 2 q^{2} & . & q^{3} & . & . & . \\
. & . & 1 & . & . & . & . & q & . & . & q^{2} + q & . & . & . & q^{3} & q^{2} & . & q^{3} & . & . & . \\
. & . & . & 1 & . & . & . & q & q & . & q^{2} & q & . & q^{2} & q^{3} & q^{3} + 2 q^{2} & . & q^{4} + q^{3} & q^{3} & q^{4} & . \\
. & . & . & . & 1 & . & . & . & . & . & q & . & . & . & q^{2} & . & . & . & . & . & . \\
. & . & . & . & . & 1 & . & . & . & . & q & . & . & . & q^{2} & . & . & . & . & . & . \\
. & . & . & . & . & . & 1 & . & . & . & . & q & q & . & . & 2 q^{2} & . & q^{3} & . & . & . \\
. & . & . & . & . & . & . & 1 & . & . & q & . & . & . & q^{2} & q & . & q^{2} & . & . & . \\
. & . & . & . & . & . & . & . & 1 & . & . & . & . & q & . & q^{2} + q & . & q^{3} & q^{2} & q^{3} & . \\
. & . & . & . & . & . & . & . & . & 1 & . & . & . & . & . & . & . & . & . & . & . \\
. & . & . & . & . & . & . & . & . & . & 1 & . & . & . & q & . & . & . & . & . & . \\
. & . & . & . & . & . & . & . & . & . & . & 1 & . & . & . & q & . & q^{2} & . & . & . \\
. & . & . & . & . & . & . & . & . & . & . & . & 1 & . & . & q & . & q^{2} & . & . & . \\
. & . & . & . & . & . & . & . & . & . & . & . & . & 1 & . & q & . & q^{2} & q & q^{2} & . \\
. & . & . & . & . & . & . & . & . & . & . & . & . & . & 1 & . & . & . & . & . & . \\
. & . & . & . & . & . & . & . & . & . & . & . & . & . & . & 1 & . & q & . & . & . \\
. & . & . & . & . & . & . & . & . & . & . & . & . & . & . & . & 1 & . & q & q^{2} & . \\
. & . & . & . & . & . & . & . & . & . & . & . & . & . & . & . & . & 1 & . & . & . \\
. & . & . & . & . & . & . & . & . & . & . & . & . & . & . & . & . & . & 1 & q & . \\
. & . & . & . & . & . & . & . & . & . & . & . & . & . & . & . & . & . & . & 1 & . \\
. & . & . & . & . & . & . & . & . & . & . & . & . & . & . & . & . & . & . & . & 1
\end{array}\right)
\end{equation}
}

\subsection{$q$-Cartan matrices for unital peak algebras of order $r=5$}

\begin{equation}
C_5^{(5)} =
\left(\begin{array}{ccccccc}
1 & q & q & q^{2} & q^{2} & q^{3} & . \\
. & 1 & . & q & . & . & . \\
. & . & 1 & . & q & q^{2} & . \\
. & . & . & 1 & . & . & . \\
. & . & . & . & 1 & q & . \\
. & . & . & . & . & 1 & . \\
. & . & . & . & . & . & 1
\end{array}\right)
\end{equation}

\begin{equation}
C_6^{(5)} =
\left(\begin{array}{ccccccccccc}
1 & . & q & . & . & q^{2} + q & q & q^{3} & q^{2} & q^{3} & . \\
. & 1 & . & . & . & q & . & q^{2} & . & . & . \\
. & . & 1 & . & . & q & . & q^{2} & . & . & . \\
. & . & . & 1 & . & q & q & q^{2} & q^{2} & q^{3} & . \\
. & . & . & . & 1 & . & . & . & . & . & . \\
. & . & . & . & . & 1 & . & q & . & . & . \\
. & . & . & . & . & . & 1 & . & q & q^{2} & . \\
. & . & . & . & . & . & . & 1 & . & . & . \\
. & . & . & . & . & . & . & . & 1 & q & . \\
. & . & . & . & . & . & . & . & . & 1 & . \\
. & . & . & . & . & . & . & . & . & . & 1
\end{array}\right)
\end{equation}

\begin{equation}
C_7^{(5)} =
\left(\begin{array}{ccccccccccccccc}
1 & q & . & q & q & q^{2} & 2 q^{2} & . & q^{2} & 3 q^{3} & q^{2} & 2 q^{4} & q^{3} & q^{4} & . \\
. & 1 & . & . & . & q & q & . & . & 2 q^{2} & . & q^{3} & . & . & . \\
. & . & 1 & . & q & . & q & . & q^{2} & q^{2} & . & q^{3} & . & . & . \\
. & . & . & 1 & . & . & q & . & . & q^{2} + q & q & q^{3} & q^{2} & q^{3} & . \\
. & . & . & . & 1 & . & . & . & q & . & . & . & . & . & . \\
. & . & . & . & . & 1 & . & . & . & q & . & q^{2} & . & . & . \\
. & . & . & . & . & . & 1 & . & . & q & . & q^{2} & . & . & . \\
. & . & . & . & . & . & . & 1 & . & q & q & q^{2} & q^{2} & q^{3} & . \\
. & . & . & . & . & . & . & . & 1 & . & . & . & . & . & . \\
. & . & . & . & . & . & . & . & . & 1 & . & q & . & . & . \\
. & . & . & . & . & . & . & . & . & . & 1 & . & q & q^{2} & . \\
. & . & . & . & . & . & . & . & . & . & . & 1 & . & . & . \\
. & . & . & . & . & . & . & . & . & . & . & . & 1 & q & . \\
. & . & . & . & . & . & . & . & . & . & . & . & . & 1 & . \\
. & . & . & . & . & . & . & . & . & . & . & . & . & . & 1
\end{array}\right)
\end{equation}

{
\tiny
\begin{equation}
C_8^{(5)} =
\left(\begin{array}{cccccccccccccccccccccc}
1 & . & . & q & q & q & q^{2} & q^{2}\! +\! q & . & q^{2} & . & q^{3}\! +\! 2
q^{2} & q^{3} & 2 q^{3} \!+\! q^{2} & . & q^{4} \!+\! q^{3} & 3 q^{4} \!+\! q^{3} & q^{3} & 2 q^{5} & q^{4} & q^{5} & . \\
. & 1 & . & . & . & . & . & q & . & . & . & . & q^{2} & q^{2} & . & . & q^{3} & . & q^{4} & . & . & . \\
. & . & 1 & . & . & q & . & q & q & . & . & 2 q^{2} & q^{2} & q^{2} & . & q^{3} & 2 q^{3} & . & q^{4} & . & . & . \\
. & . & . & 1 & . & . & q & . & . & . & . & q^{2} \!+\! q & . & q & . & q^{3} & q^{2} & . & q^{3} & . & . & . \\
. & . & . & . & 1 & . & . & q & . & q & . & q & q^{2} & 2 q^{2} & . & q^{2} & 3 q^{3} & q^{2} & 2 q^{4} & q^{3} & q^{4} & . \\
. & . & . & . & . & 1 & . & . & . & . & . & q & . & . & . & q^{2} & . & . & . & . & . & . \\
. & . & . & . & . & . & 1 & . & . & . & . & q & . & . & . & q^{2} & . & . & . & . & . & . \\
. & . & . & . & . & . & . & 1 & . & . & . & . & q & q & . & . & 2 q^{2} & . & q^{3} & . & . & . \\
. & . & . & . & . & . & . & . & 1 & . & . & q & . & q & . & q^{2} & q^{2} & . & q^{3} & . & . & . \\
. & . & . & . & . & . & . & . & . & 1 & . & . & . & q & . & . & q^{2} \!+\! q & q & q^{3} & q^{2} & q^{3} & . \\
. & . & . & . & . & . & . & . & . & . & 1 & . & . & . & . & . & . & . & . & . & . & . \\
. & . & . & . & . & . & . & . & . & . & . & 1 & . & . & . & q & . & . & . & . & . & . \\
. & . & . & . & . & . & . & . & . & . & . & . & 1 & . & . & . & q & . & q^{2} & . & . & . \\
. & . & . & . & . & . & . & . & . & . & . & . & . & 1 & . & . & q & . & q^{2} & . & . & . \\
. & . & . & . & . & . & . & . & . & . & . & . & . & . & 1 & . & q & q & q^{2} & q^{2} & q^{3} & . \\
. & . & . & . & . & . & . & . & . & . & . & . & . & . & . & 1 & . & . & . & . & . & . \\
. & . & . & . & . & . & . & . & . & . & . & . & . & . & . & . & 1 & . & q & . & . & . \\
. & . & . & . & . & . & . & . & . & . & . & . & . & . & . & . & . & 1 & . & q & q^{2} & . \\
. & . & . & . & . & . & . & . & . & . & . & . & . & . & . & . & . & . & 1 & . & . & . \\
. & . & . & . & . & . & . & . & . & . & . & . & . & . & . & . & . & . & . & 1 & q & . \\
. & . & . & . & . & . & . & . & . & . & . & . & . & . & . & . & . & . & . & . & 1 & . \\
. & . & . & . & . & . & . & . & . & . & . & . & . & . & . & . & . & . & . & . & . & 1
\end{array}\right)
\end{equation}
}

\subsection{$q$-Cartan matrices for unital peak algebras of order $r=6$}

\begin{equation}
C_6^{(6)} =
\left(\begin{array}{ccccccccccc}
1 & . & q & q & . & 2 q^{2} & q^{2} & q^{3} & q^{3} & q^{4} & . \\
. & 1 & . & . & . & q & . & q^{2} & . & . & . \\
. & . & 1 & . & . & q & . & q^{2} & . & . & . \\
. & . & . & 1 & . & q & q & q^{2} & q^{2} & q^{3} & . \\
. & . & . & . & 1 & . & . & . & . & . & . \\
. & . & . & . & . & 1 & . & q & . & . & . \\
. & . & . & . & . & . & 1 & . & q & q^{2} & . \\
. & . & . & . & . & . & . & 1 & . & . & . \\
. & . & . & . & . & . & . & . & 1 & q & . \\
. & . & . & . & . & . & . & . & . & 1 & . \\
. & . & . & . & . & . & . & . & . & . & 1
\end{array}\right)
\end{equation}

\begin{equation}
C_7^{(6)} =
\left(\begin{array}{ccccccccccccccc}
1 & q & q & . & q^{2} & q^{2} & 2 q^{2} & q & q^{3} & 3 q^{3} & q^{2} & 2 q^{4} & q^{3} & q^{4} & . \\
. & 1 & . & . & . & q & q & . & . & 2 q^{2} & . & q^{3} & . & . & . \\
. & . & 1 & . & q & . & q & . & q^{2} & q^{2} & . & q^{3} & . & . & . \\
. & . & . & 1 & . & . & q & q & . & 2 q^{2} & q^{2} & q^{3} & q^{3} & q^{4} & . \\
. & . & . & . & 1 & . & . & . & q & . & . & . & . & . & . \\
. & . & . & . & . & 1 & . & . & . & q & . & q^{2} & . & . & . \\
. & . & . & . & . & . & 1 & . & . & q & . & q^{2} & . & . & . \\
. & . & . & . & . & . & . & 1 & . & q & q & q^{2} & q^{2} & q^{3} & . \\
. & . & . & . & . & . & . & . & 1 & . & . & . & . & . & . \\
. & . & . & . & . & . & . & . & . & 1 & . & q & . & . & . \\
. & . & . & . & . & . & . & . & . & . & 1 & . & q & q^{2} & . \\
. & . & . & . & . & . & . & . & . & . & . & 1 & . & . & . \\
. & . & . & . & . & . & . & . & . & . & . & . & 1 & q & . \\
. & . & . & . & . & . & . & . & . & . & . & . & . & 1 & . \\
. & . & . & . & . & . & . & . & . & . & . & . & . & . & 1
\end{array}\right)
\end{equation}

{
\tiny
\begin{equation}
C_8^{(6)} =
\left(\begin{array}{cccccccccccccccccccccc}
1 & . & q & . & q & q^{2} & q & 2 q^{2} & 2 q^{2} & . & . & 3 q^{3} & q^{3} & 3 q^{3} & q^{2} & 2 q^{4} & 4 q^{4} & q^{3} & 2 q^{5} & q^{4} & q^{5} & . \\
. & 1 & . & . & . & . & . & q & . & . & . & . & q^{2} & q^{2} & . & . & q^{3} & . & q^{4} & . & . & . \\
. & . & 1 & . & . & q & . & q & q & . & . & 2 q^{2} & q^{2} & q^{2} & . & q^{3} & 2 q^{3} & . & q^{4} & . & . & . \\
. & . & . & 1 & . & . & q & . & q & . & . & 2 q^{2} & . & q^{2} & . & q^{3} & q^{3} & . & q^{4} & . & . & . \\
. & . & . & . & 1 & . & . & q & q & . & . & q^{2} & q^{2} & 2 q^{2} & q & q^{3} & 3 q^{3} & q^{2} & 2 q^{4} & q^{3} & q^{4} & . \\
. & . & . & . & . & 1 & . & . & . & . & . & q & . & . & . & q^{2} & . & . & . & . & . & . \\
. & . & . & . & . & . & 1 & . & . & . & . & q & . & . & . & q^{2} & . & . & . & . & . & . \\
. & . & . & . & . & . & . & 1 & . & . & . & . & q & q & . & . & 2 q^{2} & . & q^{3} & . & . & . \\
. & . & . & . & . & . & . & . & 1 & . & . & q & . & q & . & q^{2} & q^{2} & . & q^{3} & . & . & . \\
. & . & . & . & . & . & . & . & . & 1 & . & . & . & q & q & . & 2 q^{2} & q^{2} & q^{3} & q^{3} & q^{4} & . \\
. & . & . & . & . & . & . & . & . & . & 1 & . & . & . & . & . & . & . & . & . & . & . \\
. & . & . & . & . & . & . & . & . & . & . & 1 & . & . & . & q & . & . & . & . & . & . \\
. & . & . & . & . & . & . & . & . & . & . & . & 1 & . & . & . & q & . & q^{2} & . & . & . \\
. & . & . & . & . & . & . & . & . & . & . & . & . & 1 & . & . & q & . & q^{2} & . & . & . \\
. & . & . & . & . & . & . & . & . & . & . & . & . & . & 1 & . & q & q & q^{2} & q^{2} & q^{3} & . \\
. & . & . & . & . & . & . & . & . & . & . & . & . & . & . & 1 & . & . & . & . & . & . \\
. & . & . & . & . & . & . & . & . & . & . & . & . & . & . & . & 1 & . & q & . & . & . \\
. & . & . & . & . & . & . & . & . & . & . & . & . & . & . & . & . & 1 & . & q & q^{2} & . \\
. & . & . & . & . & . & . & . & . & . & . & . & . & . & . & . & . & . & 1 & . & . & . \\
. & . & . & . & . & . & . & . & . & . & . & . & . & . & . & . & . & . & . & 1 & q & . \\
. & . & . & . & . & . & . & . & . & . & . & . & . & . & . & . & . & . & . & . & 1 & . \\
. & . & . & . & . & . & . & . & . & . & . & . & . & . & . & . & . & . & . & . & . & 1
\end{array}\right)
\end{equation}
}

\subsection{$q$-Cartan matrices for unital peak algebras of order $r=7$}

\begin{equation}
C_7^{(7)} =
\left(\begin{array}{ccccccccccccccc}
1 & q & q & q & q^{2} & q^{2} & 2 q^{2} & q^{2} & q^{3} & 3 q^{3} & q^{3} & 2 q^{4} & q^{4} & q^{5} & . \\
. & 1 & . & . & . & q & q & . & . & 2 q^{2} & . & q^{3} & . & . & . \\
. & . & 1 & . & q & . & q & . & q^{2} & q^{2} & . & q^{3} & . & . & . \\
. & . & . & 1 & . & . & q & q & . & 2 q^{2} & q^{2} & q^{3} & q^{3} & q^{4} & . \\
. & . & . & . & 1 & . & . & . & q & . & . & . & . & . & . \\
. & . & . & . & . & 1 & . & . & . & q & . & q^{2} & . & . & . \\
. & . & . & . & . & . & 1 & . & . & q & . & q^{2} & . & . & . \\
. & . & . & . & . & . & . & 1 & . & q & q & q^{2} & q^{2} & q^{3} & . \\
. & . & . & . & . & . & . & . & 1 & . & . & . & . & . & . \\
. & . & . & . & . & . & . & . & . & 1 & . & q & . & . & . \\
. & . & . & . & . & . & . & . & . & . & 1 & . & q & q^{2} & . \\
. & . & . & . & . & . & . & . & . & . & . & 1 & . & . & . \\
. & . & . & . & . & . & . & . & . & . & . & . & 1 & q & . \\
. & . & . & . & . & . & . & . & . & . & . & . & . & 1 & . \\
. & . & . & . & . & . & . & . & . & . & . & . & . & . & 1
\end{array}\right)
\end{equation}

{
\tiny
\begin{equation}
C_8^{(7)} =
\left(\begin{array}{cccccccccccccccccccccc}
1 & . & q & q & . & q^{2} & q^{2} & q^{2} + q & 2 q^{2} & q & . & 3 q^{3} & q^{3} & 2 q^{3} + q^{2} & q^{2} & 2 q^{4} & 3 q^{4} + q^{3} & q^{3} & 2 q^{5} & q^{4} & q^{5} & . \\
. & 1 & . & . & . & . & . & q & . & . & . & . & q^{2} & q^{2} & . & . & q^{3} & . & q^{4} & . & . & . \\
. & . & 1 & . & . & q & . & q & q & . & . & 2 q^{2} & q^{2} & q^{2} & . & q^{3} & 2 q^{3} & . & q^{4} & . & . & . \\
. & . & . & 1 & . & . & q & . & q & . & . & 2 q^{2} & . & q^{2} & . & q^{3} & q^{3} & . & q^{4} & . & . & . \\
. & . & . & . & 1 & . & . & q & q & q & . & q^{2} & q^{2} & 2 q^{2} & q^{2} & q^{3} & 3 q^{3} & q^{3} & 2 q^{4} & q^{4} & q^{5} & . \\
. & . & . & . & . & 1 & . & . & . & . & . & q & . & . & . & q^{2} & . & . & . & . & . & . \\
. & . & . & . & . & . & 1 & . & . & . & . & q & . & . & . & q^{2} & . & . & . & . & . & . \\
. & . & . & . & . & . & . & 1 & . & . & . & . & q & q & . & . & 2 q^{2} & . & q^{3} & . & . & . \\
. & . & . & . & . & . & . & . & 1 & . & . & q & . & q & . & q^{2} & q^{2} & . & q^{3} & . & . & . \\
. & . & . & . & . & . & . & . & . & 1 & . & . & . & q & q & . & 2 q^{2} & q^{2} & q^{3} & q^{3} & q^{4} & . \\
. & . & . & . & . & . & . & . & . & . & 1 & . & . & . & . & . & . & . & . & . & . & . \\
. & . & . & . & . & . & . & . & . & . & . & 1 & . & . & . & q & . & . & . & . & . & . \\
. & . & . & . & . & . & . & . & . & . & . & . & 1 & . & . & . & q & . & q^{2} & . & . & . \\
. & . & . & . & . & . & . & . & . & . & . & . & . & 1 & . & . & q & . & q^{2} & . & . & . \\
. & . & . & . & . & . & . & . & . & . & . & . & . & . & 1 & . & q & q & q^{2} & q^{2} & q^{3} & . \\
. & . & . & . & . & . & . & . & . & . & . & . & . & . & . & 1 & . & . & . & . & . & . \\
. & . & . & . & . & . & . & . & . & . & . & . & . & . & . & . & 1 & . & q & . & . & . \\
. & . & . & . & . & . & . & . & . & . & . & . & . & . & . & . & . & 1 & . & q & q^{2} & . \\
. & . & . & . & . & . & . & . & . & . & . & . & . & . & . & . & . & . & 1 & . & . & . \\
. & . & . & . & . & . & . & . & . & . & . & . & . & . & . & . & . & . & . & 1 & q & . \\
. & . & . & . & . & . & . & . & . & . & . & . & . & . & . & . & . & . & . & . & 1 & . \\
. & . & . & . & . & . & . & . & . & . & . & . & . & . & . & . & . & . & . & . & . & 1
\end{array}\right)
\end{equation}
}

\subsection{$q$-Cartan matrices for unital peak algebras of order $r=8$}

{
\footnotesize
\begin{equation}
C_8^{(8)} =
\left(\begin{array}{cccccccccccccccccccccc}
1 & . & q & q & q & q^{2} & q^{2} & 2 q^{2} & 2 q^{2} & q^{2} & . & 3 q^{3} & q^{3} & 3 q^{3} & q^{3} & 2 q^{4} & 4 q^{4} & q^{4} & 2 q^{5} & q^{5} & q^{6} & . \\
. & 1 & . & . & . & . & . & q & . & . & . & . & q^{2} & q^{2} & . & . & q^{3} & . & q^{4} & . & . & . \\
. & . & 1 & . & . & q & . & q & q & . & . & 2 q^{2} & q^{2} & q^{2} & . & q^{3} & 2 q^{3} & . & q^{4} & . & . & . \\
. & . & . & 1 & . & . & q & . & q & . & . & 2 q^{2} & . & q^{2} & . & q^{3} & q^{3} & . & q^{4} & . & . & . \\
. & . & . & . & 1 & . & . & q & q & q & . & q^{2} & q^{2} & 2 q^{2} & q^{2} & q^{3} & 3 q^{3} & q^{3} & 2 q^{4} & q^{4} & q^{5} & . \\
. & . & . & . & . & 1 & . & . & . & . & . & q & . & . & . & q^{2} & . & . & . & . & . & . \\
. & . & . & . & . & . & 1 & . & . & . & . & q & . & . & . & q^{2} & . & . & . & . & . & . \\
. & . & . & . & . & . & . & 1 & . & . & . & . & q & q & . & . & 2 q^{2} & . & q^{3} & . & . & . \\
. & . & . & . & . & . & . & . & 1 & . & . & q & . & q & . & q^{2} & q^{2} & . & q^{3} & . & . & . \\
. & . & . & . & . & . & . & . & . & 1 & . & . & . & q & q & . & 2 q^{2} & q^{2} & q^{3} & q^{3} & q^{4} & . \\
. & . & . & . & . & . & . & . & . & . & 1 & . & . & . & . & . & . & . & . & . & . & . \\
. & . & . & . & . & . & . & . & . & . & . & 1 & . & . & . & q & . & . & . & . & . & . \\
. & . & . & . & . & . & . & . & . & . & . & . & 1 & . & . & . & q & . & q^{2} & . & . & . \\
. & . & . & . & . & . & . & . & . & . & . & . & . & 1 & . & . & q & . & q^{2} & . & . & . \\
. & . & . & . & . & . & . & . & . & . & . & . & . & . & 1 & . & q & q & q^{2} & q^{2} & q^{3} & . \\
. & . & . & . & . & . & . & . & . & . & . & . & . & . & . & 1 & . & . & . & . & . & . \\
. & . & . & . & . & . & . & . & . & . & . & . & . & . & . & . & 1 & . & q & . & . & . \\
. & . & . & . & . & . & . & . & . & . & . & . & . & . & . & . & . & 1 & . & q & q^{2} & . \\
. & . & . & . & . & . & . & . & . & . & . & . & . & . & . & . & . & . & 1 & . & . & . \\
. & . & . & . & . & . & . & . & . & . & . & . & . & . & . & . & . & . & . & 1 & q & . \\
. & . & . & . & . & . & . & . & . & . & . & . & . & . & . & . & . & . & . & . & 1 & . \\
. & . & . & . & . & . & . & . & . & . & . & . & . & . & . & . & . & . & . & . & . & 1
\end{array}\right)
\end{equation}
}



\footnotesize
\begin{thebibliography}{aa}
%
\bibitem{AR}{\sc R. Adin} and {\sc Y. Roichman},
{\it The flag major index and group actions on polynomial rings},
Europ. J. Combin, {\bf 22} (2001), 431--446.
%
\bibitem{ABN}{\sc M. Aguiar, N. Bergeron,} and {\sc K. Nyman},
{\it The peak algebra and the descent algebra of type B and D\/},
Trans. of the AMS. {\bf 356} (2004), 2781--2824.
%
\bibitem{AL}{\sc M. Aguiar} and {\sc M. Livernet},
{\it The associative operad and the weak order on the symmetric group},
J. Homotopy and Related Structures, {\bf 2} n.1 (2007), 57--84.
%
\bibitem{ANT}{\sc M. Aguiar}, {\sc J.-C. Novelli} and {\sc J.-Y. Thibon},
{\it Unital versions of the higher order peak algebras},
arXiv:0810.4634.
%
\bibitem{ANO}{\sc M. Aguiar}, {\sc K. Nyman} and {\sc R. Orellana},
{\it New results on the peak algebra},
J. Alg. Comb. {\bf 23}, 2 (2006), 149--188.
%
\bibitem{AS}{\sc M. Aguiar} and {\sc F. Sottile},
{\it Structure of the Malvenuto-Reutenauer Hopf algebra of permutations},
Adv. in Math. {\bf 191} (2005), 225--275.
%
\bibitem{BB}{\sc F. Bergeron} and {\sc N. Bergeron},
{\it Orthogonal idempotents in the descent algebra of $B_n$
and applications},
J. Pure Appl. Algbera {\bf 79} (1992), 109--129.
%
\bibitem{NBergeron1992}{\sc N. Bergeron},
{\it A decomposition of the descent algebra of the hyperoctahedral group, {II}},
J. Algebra, {\bf 148}, (1992), 98--122.
%
\bibitem{BHT}{\sc N. Bergeron, F. Hivert} and {\sc J.-Y. Thibon},
{\it The peak algebra and the Hecke-Clifford algebras at $q=0$ },
J. Combinatorial Theory A {\bf 117} (2004), 1--19.
%
\bibitem{BL}
{\sc D.Blessenohl} and {\sc H. Laue}
{\it The module structure of Solomon's descent algebra},
J. Aust. Math. Soc. {\bf 72} (2002), no. 3, 317--333.
%
\bibitem{Chow}{\sc C.-O. Chow},
{\it Noncommutative symmetric functions of type B}, Thesis, MIT, 2001.
%
\bibitem{NCSF7}{\sc G. Duchamp, F. Hivert, J.-C. Novelli},
and {\sc J.-Y. Thibon},
{\it Noncommutative Symmetric Functions VII: Free Quasi-Symmetric Functions
Revisited}, preprint, math.CO/0809.4479.
%
\bibitem{NCSF6}{\sc G. Duchamp, F. Hivert}, and {\sc J.-Y. Thibon},
{\it Noncommutative symmetric functions VI: free quasi-symmetric functions and
related algebras},
Internat. J. Alg. Comput. {\bf 12} (2002), 671--717.
%
\bibitem{GR}{\sc A. M. Garsia} and {\sc C. Reutenauer},
{\it A decomposition of Solomon's descent algebra},
Adv. in Math. {\bf 77} (1989), 189--262.
%
\bibitem{NCSF1}{\sc I.M. Gelfand, D. Krob, A. Lascoux, B. Leclerc,
V.~S. Retakh}, and {\sc J.-Y. Thibon},
{\it Noncommutative symmetric functions},
Adv. in Math. {\bf 112} (1995), 218--348.
%
\bibitem{HT} {\sc F. Hivert} and {\sc N. Thi\'ery},
{\it MuPAD-Combinat, an open-source package for research in algebraic
combinatorics},
S\'em. Lothar.  Combin. {\bf 51} (2004), 70p. (electronic).
%
\bibitem{NCSF2}{\sc D. Krob, B. Leclerc}, and {\sc J.-Y. Thibon},
{\it Noncommutative symmetric functions II: Transformations of alphabets},
Internal J. Alg. Comput. {\bf 7} (1997), 181--264.
%
\bibitem{KT}{\sc D. Krob} and {\sc J.-Y. Thibon},
{\it Higher order peak algebras},
Ann. Combin. {\bf 9} (2005), 411--430.
%
\bibitem{Mcd}{\sc I.G. Macdonald},
{\it Symmetric functions and Hall polynomials},
2nd ed., Oxford University Press, 1995.
%
\bibitem{MR}{\sc R. Mantaci} and {\sc C. Reutenauer},
{\it A generalization of Solomon's descent algebra for hyperoctahedral groups
and wreath products},
Comm. Algebra {\bf 23} (1995), 27--56.
%
\bibitem{NT}{\sc J.-C. Novelli} and {\sc J.-Y. Thibon},
{\it Free quasi-symmetric functions of arbitrary level},
preprint
math.CO/0405597.
%
\bibitem{NT-super}{\sc J.-C. Novelli} and {\sc J.-Y. Thibon},
{\it Superization and $(q,t)$-specialization in combinatorial Hopf algebras},
math.CO/0803.1816.
%
\bibitem{Poi}{\sc S. Poirier},
{\it Cycle type and descent set in wreath products},
Disc. Math., {\bf 180} (1998), 315--343.
%
\bibitem{Reu}{\sc C. Reutenauer},
{\it Free Lie algebras},
Oxford University Press, 1993.
%
\bibitem{Sal}{\sc F. Saliola},
{\it On the quiver of the descent algebra},
J. Algebra, {\bf 320}, (2008), pp. 3866-3894.
%
\bibitem{So}{\sc L. Solomon},
{ \it A Mackey formula in the group ring of a Coxeter group},
J. Algebra, {\bf 41}, (1976), 255-268.
%
\bibitem{Stan2}{\sc R. P. Stanley},
{\it Enumerative combinatorics},
vol. 2, Cambridge University Press, 1999.
%
\bibitem{Sage}{\sc W.\thinspace{}A. Stein} \emph{et~al.}, 
{\it {S}age {M}athematics {S}oftware ({V}ersion 3.3)}, 
The Sage~Development Team, 2009, {\tt http://www.sagemath.org}.
%
\end{thebibliography}
\end{document}